\documentclass[12pt,psfig]{article}
\usepackage{graphics,epsfig,amssymb}

\usepackage{amsmath}
\usepackage{amsthm}
\usepackage{amsfonts}
\usepackage{amssymb,latexsym}
\usepackage{graphicx}
\usepackage[mathscr]{eucal}
\usepackage{verbatim}
\usepackage{hyperref}
\usepackage{color}

\theoremstyle{plain}
\newtheorem{thm}{Theorem}[section]
\newtheorem{prop}{Proposition}[section]
\newtheorem{lemma}{Lemma}[section]

\newtheorem{rem}{Remark}[section]

\numberwithin{equation}{section}

\usepackage{graphicx}

\begin{document}
\title{Spatio-temporal dynamics for a class of monotone evolution  systems}
\author{Taishan Yi\\
School of Mathematics (Zhuhai)\\
Sun Yat-Sen University\\
Zhuhai, Guangdong 519082, China
\and
Xiao-Qiang Zhao\thanks{The corresponding author.
	E-mail: zhao@mun.ca}\\
Department of Mathematics and Statistics\\
Memorial University of Newfoundland\\
St. John's, NL A1C 5S7, Canada}

\date {}
\maketitle

\begin{abstract}
In this paper,  under an abstract setting we establish  the spreading properties
and the existence, non-existence and global attractivity of spatially
heterogeneous steady states  for a large class of monotone
evolution systems without the translational
monotonicity  under the assumption that one limiting system has
both leftward and rightward  spreading speeds  and the
other one has the uniform asymptotic annihilation. Then we apply the
developed theory to study the global  dynamics  of asymptotically homogeneous  integro-difference equations, and  provide a counter-example to show that
 the  value of the nonlinear function  at the finite range of location  may
 give rise to  nontrivial fixed points.
\end{abstract}


\smallskip


\baselineskip=18pt

\section {Introduction}
Since the pioneering works of Fisher \cite{f1937} and Kolmogorov et al. \cite{kpp1937},  extensive studies have been conducted on traveling wave solutions and propagation properties of various evolution systems with spatial structures, see, e.g., \cite{aw1975,aw1978,LYZ,l1989} and references therein.
From the perspective of dynamical systems, Weinberger \cite{w1982} developed a theoretical framework for traveling waves and spreading speeds in monotone discrete-time systems possessing spatial translation invariance. This framework was later substantially extended to monotone discrete and continuous-time semiflows in \cite{lz2010,lz2007,w2002}, as well as to certain non-monotone systems in \cite{hz2008, yz2025,yz2015}, thereby enabling its application to a wide range of evolution systems in both homogeneous and periodic environments.

Motivated by ecological systems under climate changes,  an increasing number of studies have focused on traveling waves and the asymptotic dynamics of evolution equations in shifting habitats, see, e.g., \cite{flw2016,FangPengZhao2018,hwz2024,jyz2024,lwz2018,pl2004,zhang2017persistence,ZhangZhao2019,zyc2024} and references therein.
In the study of traveling  waves and spreading properties for monotone  systems,
it is often assumed that the solution maps admit the  spatial translation
invariance (see, e.g., \cite{lz2007,w1982}). Returning to an evolution equation with
spatial variable $x\in \mathbb{R}$,  this characterizes the property that if $u(t,x)$ is a solution, then so is
$u(t,x-y)$ for any $y\in  \mathbb{R}$.  As a prototypical  example  of systems
without the translation invariance,  we  consider a scalar reaction-diffusion equation with a shifting habitat  (see, e.g., \cite{bdnz2009,Scheel,lbsf2014}):
\begin{equation}\label{shift}
	u_t=d \Delta u+f(x-ct,u), \quad  x\in\mathbb{R}, \,    \, t\ge 0,
\end{equation}
where $d>0$, $c \in\mathbb{R}$ is the shifting speed, and  $f$ is a continuous function such that $f(z,0)=0$ fo all $z\in\mathbb{R}$.
Letting $z=x-ct$ and $v(t,z)=u(t,z+ct)$,  we then transform equation \eqref{shift} into
the following spatially heterogeneous  reaction-diffusion-advection equation:
\begin{equation}\label{shift2}
	v_t=d \Delta v +cv_z+f(z,v), \quad  z\in\mathbb{R}, \,    \, t\ge 0.
\end{equation}
Let $\Phi_t$ be the time-$t$ map of equation \eqref{shift2}. Since
$f(z,v)$ depends on $z$,  it is easy to see that the map $\Phi_t$ does not
admit the translation invariance.

Assume that $\lim_{z\to \pm \infty} f(z,v)=f_{\pm}^{\infty}(v)$ uniformly for $v$ in any bounded subset of
$\mathbb{R}_{+}$.  Then equation \eqref{shift2}  has the following two limiting equations:
\begin{equation}\label{limitE}
	v_t=d \Delta v +cv_z+f_{\pm}^{\infty}(v), \quad  z\in\mathbb{R}, \,    \, t\ge 0.
\end{equation}
Let  $\Phi_t^{\pm}$ be the time-$t$ map of equation \eqref{limitE}.  Clearly,
$\Phi_t^{\pm}$ admits the translation invariance in the sense that
$\Phi_t^{\pm}\circ T_y=T_y\circ \Phi_t^{\pm}$ for all  $y\in \mathbb{R}$.
If, in addition, $f(z,v)$ is nondecreasing in $z\in\mathbb{R}$,  then we can verify that each $\Phi_t$ has the following translational  monotonicity:
\begin{enumerate}
	\item[(\bf TM)]  $T_{-y}\circ  \Phi_t[\varphi]\geq \Phi_t\circ T_{-y}[\varphi],
	\quad     \forall \varphi\in BC(\mathbb{R},\mathbb{R}_+),\,    \,
	y\in  \mathbb{R}_+$.
\end{enumerate}
Since systems \eqref{shift2} and \eqref{limitE} admit the comparison principle,
it follows that  their time-$t$ maps $\Phi_t$ and  $\Phi_t^{\pm}$ are
order preserving (i.e., monotone).

Under an abstract setting,  we  studied the  propagation dynamics for a large class of monotone evolution systems without translation invariance in \cite{yz2020}. More precisely, we established the spreading properties
and the existence of  nontrivial steady states for monotone semiflows
with the property ({\bf TM}) under the assumption that one limiting system has
both leftward and rightward  spreading speeds (good property G) and the
other one has the uniform asymptotic annihilation (bad property B).
The global dynamics of such monotone systems with the BB  and GG
combinations were further investigated in our subsequent papers \cite{yz2023,yz2025},
respectively.  Here we should point out that the nonexistence of nontrivial fixed points
were not addressed in \cite{yz2020}.

The purpose of our current paper is to investigate the propagation dynamics
for monotone systems without the assumption  ({\bf TM}) in the case where
two limiting systems have good property  G  and bad property  B
(i.e., the GB combination), respectively. This greatly extends the application range of the
developed theory to population models in ecology and epidemiology. For example,
letting $f(x,u)=u(r(x)-u)$ in equation \eqref{shift}  (i.e., in the Fisher-KPP case), we do not need  to  request that the resource function $r(x)$  be nondecreasing in $x\in \mathbb{R}$
after assuming $-\infty<r(-\infty)<0< r(\infty) <\infty$.
To overcome the difficulty induced by the lack of property ({\bf TM}),
we introduce the assumption of  linear upper systems. We also identify an  integro-difference equation to  show that the existence of nontrivial fixed points may
depend on both intermediate and limiting systems.

The rest of this paper is organized as follows.  In Section 2,  we prove the upward convergence and asymptotic annihilation of solutions,
and the existence, non-existence and global attractivity
of nontrivial fixed points for discrete-time semiflows under  the assumption of  {\bf(UC)} and {\bf(UAA)} combination for two  limiting systems.
In Section 3, we extends these results to  continuous-time  semiflows and nonautonomous systems, respectively.
In Section 4, we apply our developed theory  in Section 2 to an integro-difference equation
for its global dynamics.  We also construct a counter-example to show that even if $c_-^*<0$ and $Q$ satisfies  {\bf (UAA)} at $-\infty$, the intermediate system can give rise to nontrivial fixed points.  In the Appendix, we provide a sufficient condition for a class of noncompact maps to admit the orbit precompactness respect to the compact open topology (i.e., the local uniform convergence topology).

\section{Discrete-time  semiflows}
Let $\mathbb{Z}$,
$\mathbb{N}$, $\mathbb{R}$, $\mathbb{R}_+$, $\mathbb{R}^N$, $\mathbb{R}_+^N$, and $\mathbb{R}^{N\times N}$ be the sets of all
integers, nonnegative integers, reals,   nonnegative reals,  N-dimensional real vectors,  N-dimensional  nonnegative real vectors, and $N\times N$-real matrices, respectively. We equip  $\mathbb{R}^N$ with the norm $||\xi||_{\mathbb{R}^N}=
\left(\sum \limits_{n=1}^{N}\xi_n^2\right)^{\frac 12}$.  Let $X=BC(\mathbb{R},\mathbb{R}^N)$ be the normed
vector space of all bounded and continuous functions from
$\mathbb{R}$ to $\mathbb{R}^N$ with the norm $||\phi||_{X}=
\sum \limits_{n=1}^{\infty}2^{-n}\sup\limits_{|x|\leq n}\{||\phi(x)||_{\mathbb{R}^N}\}$. Let $X_+=\{\phi\in
X:  \,   \phi(x)\in \mathbb{R}_+^N, \,  \forall x\in \mathbb{R}\}$ and
$X_+^\circ=\{\phi\in X:   \, \phi(x)\in Int(\mathbb{R}_+^N),
\,  \forall  x\in \mathbb{R}\}$.

For a given compact topological space $M$,  let $C=C(M,X)$ be the
normed vector space of all continuous functions from $M$
into $X$ with the norm $||\varphi||_{C}=
\sup\limits_{\theta\in M}\{||\varphi(\theta)||_{X}\}$,  $C_+=C(M,X_+)$ and
$C_{+}^{\circ}=C(M,X_{+}^\circ)$. It then follows that $C_+$ is a closed
cone in the normed vector space $C$.  Let
$Y=C(M,\mathbb{R}^N)$ be the normed vector space of all
continuous functions from $M$ into $\mathbb{R}^N$ with the
norm $||\beta||_{Y}=  \sup\limits_{\theta\in
M} \{||\beta(\theta)||_{\mathbb{R}^N}\}$ and
$Y_{+}=C(M,\mathbb{R}_+^N)$.  Clearly,  $Int(Y_+)=\{\beta\in Y: \, \beta (\theta)\in Int(\mathbb{R}_+^N), \,  \forall  \theta\in M\}$.

For the sake of convenience, we identify an element $\varphi\in C$
with a bounded and continuous function from $M\times
\mathbb{R}$ into $\mathbb{R}^N$. For $a \in \mathbb{R}^N$, $\hat{a}\in
X$ is defined as ${\hat{a}}(x)=a$ for all $x\in \mathbb{R}$.
Similarly, $\hat{\hat{a}}\in C$ is defined as
$\hat{\hat{a}}(\theta)=\hat{a}$ for all $\theta\in M$.
Moreover, for any  $\phi\in X$ and $\beta\in Y$, we define
$\tilde{\phi}\in C$ and $\tilde{\beta}\in C$, respectively,  by
$\tilde{\phi}(\theta,x)=\phi(x)$  and
$\tilde{\beta}(\theta,x)=\beta(\theta)$ for all $(\theta,x)\in
M\times \mathbb{R}$. In the following, we always  identify
$\hat{a}$ or $\hat{\hat{a}}$ with $a$ for $a\in \mathbb{R}^N$.
Furthermore, we identify $\phi\in X$ and $\beta\in Y$ with
$\tilde{\phi}\in C$ and $\tilde{\beta}\in C$, respectively. Accordingly, we can regard  both $X$ and $Y$ as subspaces of $C$.

 For any $\xi$, $\eta \in \mathbb{R}^N$,  we write $\xi\geq_{\mathbb{R}^N} \eta$ if $\xi-\eta
 \in \mathbb{R}^N_+$; $\xi >_{\mathbb{R}^N} \eta$ if $\xi\geq_{\mathbb{R}^N} \eta$ and $\xi\neq \eta$;
 $\xi \gg_{\mathbb{R}^N} \eta$ if $\xi- \eta \in Int(\mathbb{R}^N_+)$.
 Similarly, we use $Y_+$ and $Int(Y_+)$ to define $\geq_Y$, $>_Y$ and $\gg_Y$ for the  space $Y$.
 For any $\xi$, $\eta \in X$,  we write $\xi\geq_X \eta$ if $\xi-\eta
 \in X_+$; $\xi >_X \eta$ if $\xi\geq_X \eta$ and $\xi\neq \eta$;
 $\xi \gg_X \eta$ if $\xi- \eta \in X_+^\circ$.
We also employ  $C_+$ and $C_+^\circ$ to define $\geq_C$ $>_C$, $\gg_C$ for the space $C$ in a similar way. For simplicity, we  write
$\geq$, $>$, $\gg$, and $||\cdot||$, respectively, for $\geq_*$,
$>_*$, $\gg_*$, and $||\cdot||_*$, where $*$ stands  for one of ${\mathbb{R}^N}$, $X$,
$Y$ and $C$.

For any given $s,r\in Int(\mathbb{R}_+^N)$ with $s\geq r$, define $C_r=\{\varphi\in
C:0\leq\varphi\leq r\}$ and $C_{r,s}=\{\varphi\in C:r\leq\varphi\leq
s \}$. For any given $\varphi\in C_+$, define $C_{\varphi}=\{\psi\in C:
0\le \psi\le \varphi\}$.
We also  define $[\varphi,\psi]_*=\{\xi\in *:\varphi\leq_*\xi\leq_*\psi\}$ and  $[[\varphi,\psi]]_*=\{\xi\in *:\varphi\ll_*\xi\ll_*\psi\}$ for $\varphi,\psi\in *$ with $\varphi\leq_*\psi$, where $*$ stands  for one of ${\mathbb{R}^N}$, $X$,
$Y$ and $C$.  Let $\check{1}:=(1,1,\cdots,1)^{T}\in \mathbb{R}^N$ and $\check{\bf 1}\in \mathbb{R}^{N\times N}$ with $(\check{\bf 1})_{ij}=1, \forall i,j\in \{1,2,\cdots,N\}$.
For any given $y\in\mathbb{R}$, we define the spatial translation operator
$T_{y}$ on $C$ by
$$T_{y}[\varphi](\theta,x)=\varphi(\theta,x-y),\quad \forall \varphi\in C,\, \theta\in M, \, x\in\mathbb{R}.
$$
We also define the reflection operator $\mathcal{S}$ on
$C$ by
$$[\mathcal{S}(\varphi)](\theta,x)= \varphi(\theta,-x),   \, \forall
(\theta,x) \in M \times  \mathbb{R},  \, \varphi\in C.
$$
As a convention, we say a  map ${Q}:  C_+\to C_+$ is continuous
if for any given  $r\in Int(\mathbb{R}_+^N)$, $Q|_{C_r}:C_r\to C_{+}$ is continuous in the usual sense.  The map ${Q}:  C_+\to C_+$ is said to be monotone if ${Q}[\phi]\leq {Q}[\psi]$ whenever $\phi,\psi\in  C_{+}$ with
$\phi\leq \psi$.

Throughout the whole paper,  unless specified otherwise, we always assume  that the map ${Q}:  C_+\to C_+$ is continuous and monotone, and  that
\begin{enumerate}
%\item [{\bf (UA)}] There exist $r^*\in Int(Y_+)$ and a continuous map $Q_+:  C_+\to C_+$ such that $Q_+[r^*]=r^*$ and
%$\lim\limits_{k\to\infty}T_{-y_k}\circ Q^n \circ T_{y_k}[\varphi_k]=Q_+^n[\varphi]$  in $C$ for any $n\in \mathbb{N}$, $y_k\to \infty$,  and $\varphi_k \to \varphi$ in $C_+$.

\item [{\bf (A)}] There exist $r^*\in Int(Y_+)$ and a continuous map $Q_+:  C_+\to C_+$ such that $Q_+[r^*]=r^*$ and
$\lim\limits_{y\to\infty}T_{-y}\circ Q^n \circ T_y[\varphi]=Q_+^n[\varphi]$  in $C$ for all $\varphi\in C_+$ and $n\in \mathbb{N}$.

\item [{\bf (UB)}]  There exists a sequence $\{\phi_k^*\}_{k\in \mathbb{N}}$ in $Int(Y_+)$ such that $\phi_{k+1}^*>\phi_k^*$, $Y_+=\bigcup\limits_{k\in \mathbb{N}}[0,\phi_k^*]_{Y}$, and $Q[C_{\phi_k^*}]\subseteq C_{\phi_k^*}$  for all $k\in \mathbb{N}$.
\end{enumerate}

In this section, we study the upward convergence, asymptotic annihilation, and
nontrivial  fixed points for discrete-time semiflows under  the assumption of  {\bf(UC)} and {\bf(UAA)} combination for the limiting system.  Following  \cite{yz2025}, we need the several hypotheses about  the  strong positivity, upward convergence, and uniform asymptotic annihilation for  the  maps $Q$ and  $Q_+$.

\begin{enumerate}

\item [{\bf (SP)}]  There exist $(\rho^*,N^*)\in (0,\infty)\times \mathbb{N}$ and  $\varrho^*\in (\rho^*,\infty)$  such that for any  $a\in \mathbb{R}$  and $\varphi\in
 C_{+}\setminus \{0\}$ with   $\varphi(\cdot, a)\subseteq Y_+\setminus \{0\}$,
we have  $Q^n[\varphi](\cdot,x)\in Int(Y_+)$ for all  $n\geq N^*$ and $x\in [a+n\rho^*,a+n\varrho^*]$.

\item [{\bf (UC)}]  There exist $c_-^*,c_+^*\in \mathbb{R}$  such that $c_+^*+c_-^*>0$ and
$$
\lim\limits_{n\rightarrow \infty}
\max\limits_{x\in \mathcal{A}_{\varepsilon,n}^+}
||Q_+^n[\varphi](\cdot,x)-r^*(\cdot)||=0, \, \forall \varepsilon\in (0,\frac{c_+^*+c_-^*}{2}), \, \varphi\in
C_{+}\setminus \{0\},
$$
 where $\mathcal{A}_{\varepsilon,n}^+=n{[-c_-^*+\varepsilon,c_+^*-\varepsilon]}$.

\item [{\bf(GUAA$_-$)}] There exist a sequence $\{(r^{**}_l,\widehat{Q}_l,\widehat{Q}_l^+)\}_{l\in \mathbb{N}}$  such that for any $l\in \mathbb{N}$,
 there hold
\begin{enumerate}
\item [(i)]   $r^{**}_l\in Int(Y_+)$  and $\widehat{Q}_l,\widehat{Q}_l^+:C_+\to C_+ $ satisfy  $C_+=\bigcup\limits_{l\in\mathbb{N}}C_{r^{**}_l}$ and
$$
Q[\varphi]\leq \widehat{Q}_l[\varphi]\leq \widehat{Q}_l[\psi]\leq \widehat{Q}_l[r^{**}_l]\leq r^{**}_l, \, \forall 0\leq \varphi\leq \psi \in C_{ r^{**}_l}.
$$

\item [(ii)]    $\lim\limits_{n\to \infty}||(\widehat{Q}_l^+)^n[r^{**}_l](\cdot,0)||=0$ and
$$
\liminf\limits_{y\to -\infty} \Big[\inf \{(\widehat{Q}_l^+)^n[r^{**}_l](\theta,0)-T_{-y}\circ (\widehat{Q}_l)^n
\circ T_y[r^{**}_l](\theta,0):\theta\in M\}\Big] \in \mathbb{R}_+^N,
\, \forall n\in \mathbb{N}.
$$
\end{enumerate}
\end{enumerate}

Let  ${\bf {1}}_{\mathbb{R}_+}:\mathbb{R} \to \mathbb{R}$ be defined by ${\bf {1}}_{\mathbb{R}_+}(x)=1, \forall x\in \mathbb{R}_+$, and   ${\bf {1}}_{\mathbb{R}_+}(x)=0, \forall x\in (-\infty,0)$.
The subsequent upward convergence result significantly  improves  \cite[Theorem 3.1] {yz2025}
by removing  the  assumption of translational monotonicity.

\begin{thm} \label{thm2.1}
Let   {\bf (SP)},  {\bf (UC)} and {\bf (GUAA$_-$)}  hold. Then for any given $\varphi\in C_{+}\setminus\{0\}$,
the following statements are valid:
\begin{itemize}

\item [{\rm (i)}]  If  $\min\{c_-^*,c_+^*\}>0$, then  $$\lim\limits_{\alpha \rightarrow \infty}
\Big[\sup\{||Q^n[\varphi](\cdot, x)-r^*\cdot {\bf {1}}_{\mathbb{R}_+}(x)||:(n,x)\in \mathcal{
B}_{\alpha,\varepsilon}^+\}\Big]= 0,$$ where $\varepsilon\in (0,c_+^*)$ and $\mathcal{B}_{\alpha,\varepsilon}^+=\{(n,x)\in \mathbb{N}\times \mathbb{R}:n\geq \alpha \mbox{ and } x \in [\alpha, n(c_+^*-\varepsilon)] \bigcup (-\infty,-\alpha]\}$.

\item [{\rm (ii)}]  If  $c_+^*>0$, then  $$\lim\limits_{n \rightarrow \infty}
\Big[\sup\{||Q^n[\varphi](\cdot, x)-r^*\cdot {\bf {1}}_{\mathbb{R}_+}(x)||:x\in n \mathcal{B}_{\varepsilon}^+\}\Big]= 0,$$ where $\varepsilon\in (0,\frac{1}{2}\min\{c_+^*,c_+^*+c_-^*\})$ and $\mathcal{B}_{\varepsilon}^+= [\max\{\varepsilon,-c_-^*+\varepsilon\}, (c_+^*-\varepsilon)] \bigcup (-\infty,-\varepsilon]$.

\end{itemize}
\end{thm}

\begin{proof}
By applying \cite[Lemma 2.3 and  Theorem 3.1-(ii)]{yz2025}  with $c=0$ to $Q$, we have
 $$
 \lim\limits_{\alpha \rightarrow \infty}
\Big[\sup\{||Q^n[\varphi](\cdot, x)-r^*||:n\geq \alpha \mbox{ and } \alpha \leq x \leq n(c_+^*-\varepsilon)\}\Big]= 0, \quad  \forall  \varepsilon\in (0,c_+^*).
$$
On the other hand, by {\bf(GUAA$_-$)}, there exists $l_{0}\in \mathbb{N}$ such that $\varphi\in C_{r^{**}_{l_{0}}}$.
By \cite[Theorem 3.2] {yz2023}, as applied to $(\mathcal{S}\circ Q \circ\mathcal{S},\mathcal{S}\circ \widehat{Q}_{l_{0}}\circ\mathcal{S},\mathcal{S}\circ \widehat{Q}_{l_{0}}^+\circ\mathcal{S},r^{**}_{l_{0}})$,
it follows  that
$$
\lim\limits_{\alpha\rightarrow \infty}
\Big[\sup\{
||{Q}^n[\varphi](\cdot,x)||:x\in  (-\infty,-\alpha] \mbox{ and } n\ge \alpha \}\Big]=0.
$$
 This proves the statement (i).
Similarly, by using \cite[Theorem~3.2-(ii)]{yz2025} and the proof of (i),  we can easily prove
the statement (ii).
\end{proof}

For any given $(\varphi,\mu)\in L^{\infty}(M\times \mathbb{R},\mathbb{R}^{N})\times \mathbb{R}$, we define a function
$$
e_{\varphi,\mu}(\theta,x)=\varphi(\theta,x) e^{-\mu x}, \,  \forall (\theta,x)\in  M\times \mathbb{R}.
$$
Let   $ L,L_+$ and $L_-$ be monotone  linear operators on the space of functions
$$
\tilde{C}:=\left\{\sum\limits_{k=1}^l e_{\varphi_k,\mu_k}:\, l\geq 1,
(\varphi_k,\mu_k)\in L^{\infty}(M\times \mathbb{R},\mathbb{R}^{N})
\times \mathbb{R}, \,  \forall 1\leq k\leq l\right\}
$$
equipped with the pointwise ordering. Here and after, let us denote the set of all bounded maps from $M\times \mathbb{R}$ into $\mathbb{R}^{N}$ by $L^{\infty}(M\times \mathbb{R},\mathbb{R}^{N})$ with the pointwise ordering induced by $L^{\infty}(M\times \mathbb{R},\mathbb{R}^{N}_+)$.
For any  $\mu\in \mathbb{R}$ with $L_+[e_{\check{1},\mu}](\cdot,0)\in Y_+$, we  define $L_{+,\mu}:Y\to Y$ by
$$
L_{+,\mu}[\phi](\theta)=L_+[e_{\phi,\mu}](\theta,0),  \,   \forall \theta\in  M,\,
\phi\in  Y.
$$
Assume that $I_{L_+}:=\{\mu\in (0,\infty):L_+[e_{\check{1},\mu}](\cdot,0)\in Y_+\}\neq \emptyset$,
$$
T_{-z}\circ L_+ \circ T_z[e_{\zeta,\mu}]=L_+[e_{\zeta,\mu}], \quad \forall  (z,\mu,\zeta)\in \mathbb{R}\times I_{L_+}\times Y_+,
$$
and   there exists
$l_0\in \mathbb{N}$ such that
$(L_{+,\mu})^{l_0}$ is a compact and strongly positive operator on $Y$ for all  $\mu\in I_{L_+}$. It then follows that
$L_{+,\mu}$ has the principal eigenvalue $\lambda(\mu)$ and  a unique  strongly
positive eigenfunction $\zeta_\mu$ associated with  $\lambda(\mu)$ such that
$\inf\{\zeta_\mu^i(\theta):\theta\in M, 1\leq i\leq N
\}=1$.  Let
$${c^*_{L_+}}:=\inf\limits_{\mu\in I_{L_+}}\frac{1}{\mu}\log\lambda(\mu)$$
and define
$$
\underline{L}[\varphi;\alpha,\mu](\theta,x)=\min\{\alpha \zeta_\mu(\theta),L[\varphi](\theta,x)\},
 \,  \forall (\varphi,\theta,x)\in \tilde{C}\times M\times \mathbb{R},  \,   (\alpha,\mu)\in \mathbb{R}_+^2.
$$

Without the translational monotonicity assumption, the limit system on the {\bf (UC)} side is no longer the upper system of the original system. Therefore, the asymptotic annihilation condition {\bf (AA)} in~\cite{yz2020} for the limit system may not imply the asymptotic annihilation property of the original system. This motivates us to introduce  the following  assumptions on upper systems.
\begin{description}
\item [{\bf (LC)}]  For any $\epsilon>0$ and $\mu>0$, there exists $ \mathfrak{x}_{\epsilon,\mu}>0$ such that
$$
L[e_{\zeta_\mu,\mu}](\theta,x)\leq (1+\mu \epsilon) L_+[e_{\zeta_\mu,\mu}](\theta,x),  \,  \forall  \theta\in M,  \,  x\geq \mathfrak{x}_{\epsilon,\mu}.
$$

\item [{\bf (GLC)}]  There exists a sequence of linear operators $\{(L_k,L_{k,+})\}_{k\in \mathbb{N}}$   with  {\bf (LC)}
such that
$$
Q[\varphi]\leq L_k[\varphi],  \, \forall \varphi\in C_{\phi_k^*},  \, k\in \mathbb{N}.
$$
\end{description}

By the proof of   Theorem \ref{thm2.1} under {\bf(GUAA$_-$)} and  \cite[Theorem 3.3]{yz2025}, as applied to $(Q, \varphi)$, we have the following result on the asymptotic annihilation under  the  {\bf (GLC)} and {\bf(UAA)}  combination.

\begin{thm} \label{thm2.2} The following statements are valid:
\begin{itemize}
\item [{\rm (i)}] Let {\bf (GLC)}  and {\bf(GUAA$_-$)} hold, and  $\liminf\limits_{k\to \infty}c^*_{{L}_k}\geq 0$. If  $\varphi\in { C_{+}}$ and
$\varphi(\cdot,x)=0$  for all sufficiently  positive $x$, then $$\lim\limits_{n\rightarrow \infty}
\Big[\sup\{
||Q^n[\varphi](\cdot,x)||: x\geq  n(\liminf\limits_{k\to \infty}c^*_{{L}_k}+\varepsilon) \mbox{ or } x\leq - n\varepsilon\}\Big]=0,
\forall \varepsilon>0.$$

\item [{\rm (ii)}] Let  {\bf (GLC)}  and {\bf(GUAA$_-$)} hold,  and  $\liminf\limits_{k\to \infty}c^*_{{L}_k}<0$. If  $\varphi\in { C_{+}}$ and
$\varphi(\cdot,x)=0$ is for all sufficiently positive $x$, then $\lim\limits_{\alpha\rightarrow \infty}\Big[\sup\{||Q^n[\varphi](\cdot,x)||:|x|,n\geq \alpha \mbox{ with } n \in \mathbb{N} \}\Big]=0$ and $$\lim\limits_{n\rightarrow \infty}\Big[\sup\{
||Q^n[\varphi](\cdot,x)||:|x|\geq n\varepsilon \}\Big]=0,
\forall \varepsilon>0.$$
\end{itemize}
\end{thm}

Recall that $\phi$ is a  nontrivial  fixed point  of the map $Q$
 if $\phi\in C_{+}\setminus \{0\}$  and
$Q[\phi]=\phi$. We say that $\phi(\theta,x)$ connects  $\xi \in Y_+$ to $\eta
\in Y_+$ provided that
\[
\phi(\cdot,-\infty):=\lim\limits_{s\to-\infty}\phi(\cdot,s)=\xi, \qquad   \text{and}\qquad  \phi(\cdot,\infty):=\lim\limits_{s\to\infty}\phi(\cdot,s)=\eta.
\]
The subsequent result is on  the existence of fixed points  under the bilateral {\bf(UC)} and {\bf(UAA)} combination.

\begin{thm} \label{thm2.3}
Assume that  $c_-^*>0$, {\bf (SP)},  {\bf (UC)}, {\bf (GUAA$_-$)} hold, and
  $\{Q^n [\phi_1^*]:n\in \mathbb{N}\}$ is precompact in $C$.  If $c_+^*>0$, then
 $Q$ has a  nontrivial  fixed point $W$ in $C_{\phi_1^*}$ connects $0$ to $r^*$. If $c_+^*\leq 0$, $T_{-y}\circ Q[\varphi]\geq Q\circ T_{-y}[\varphi]$ for all $(y,\varphi)\in \mathbb{R}_+\times C_{\phi_1^*}$, and if  there exists   $z:=z_l\in (-c_{+}^*,c_{-}^*)\cap \mathbb{R}_+$ such that $\{(T_z\circ Q)^n [\phi_1^*]:n\in \mathbb{N}\}$ are precompact in $C$, then  $Q$ has a  nontrivial  fixed point $W$ in $C_{\phi_1^*}$ connects $0$ to $r^*$.
\end{thm}

\begin{proof}
By applying  \cite[Proposition 3.1]{yz2025} to $Q$, we see that $Q$ has a  nontrivial  fixed point $W$ in $C_{\phi_1^*}$ such that $W(\cdot,\infty)=r^*$.  Thus, by using the proof of  Theorem~\ref{thm2.1} under {\bf(GUAA$_-$)}, we further obtain $W(\cdot,-\infty)=0$.
\end{proof}

Note that in  \cite{yz2020} we did not address the non-existence of the nontrivial fixed points. To pride a general result in this direction, we  assume that  the linear operator  $\underline{\mathcal{L}_+}$ on $\tilde{C}$ preserves  the pointwise ordering and define $\underline{\mathcal{L}_{+,\mu}}:Y\to Y$
by $\underline{\mathcal{L}_{+,\mu}}[\phi]=\underline{\mathcal{L}_+}[e_{\phi,-\mu}](\cdot,0)$ with $\underline{\mathcal{L}_{+,\mu}}$ being  continuous and positive, whenever  $\mu\in I^-_{\underline{\mathcal{L}_+}}:=\{\mu>0:\underline{\mathcal{L}_+}[e_{\check{1},-\mu}](\cdot,0)\in Y_+\}$.  Assume that $(\underline{\mathcal{L}_{+,\mu}})^{l_0}$ is a compact and strongly positive operator on $Y$ for some $l_0\in \mathbb{N}$ and  all $\mu\in I^-_{\underline{\mathcal{L}_+}}$. It  then follows that $\underline{\mathcal{L}_{+,\mu}}$ has the principal eigenvalue $\underline{\lambda}(\mu)$ and  a unique strongly
positive eigenfunction $\underline{\zeta_\mu}$ associated with  $\underline{\lambda}(\mu)$ such that $\inf\{\underline{\zeta_\mu}^i(\theta):\theta\in M,  i\in [1,N]\cap\mathbb{N}\}=1$.
 %and $\underline{\mathcal{L}}[\varphi;\alpha,\mu](\theta,x)=\min\{\alpha \zeta_\mu(\theta),L[\varphi](\theta,x)\}$ for all $(\varphi,\theta,x)\in \tilde{C}\times M\times \mathbb{R}$.

\begin{lemma}   \label{lemma2.1}
 Assume that $ T_{-y}\circ \underline{\mathcal{L}_{+}}[e_{\zeta,-\mu}]= \underline{\mathcal{L}_{+}}\circ T_{-y}[e_{\zeta,-\mu}]$ for all $(y,\mu,\zeta)\in \mathbb{R}_+^2\times Y_+$ with $\underline{\mathcal{L}_+}[e_{\check{1},-\mu}](\cdot,0)\in Y_+$.  If ${c^*_{-}(\underline{\mathcal{L}_{+}})}:=\inf\limits_{\mu\in I^-_{\underline{\mathcal{L}_+}}}\frac{1}{\mu}\log\underline{\lambda}(\mu)<0$, then for any $\varepsilon\in\left(0,-c^*_-(\underline{\mathcal{L}_{+}})\right)$, there exists a positive number $\mu:=\mu_\varepsilon$ such that
	\begin{itemize}
		\item[(i)] $\underline{\lambda}(\mu)<e^{(c^*_-(\underline{\mathcal{L}_{+}})+\varepsilon)\mu}<1$.
				
		\item[(ii)] $\underline{\mathcal{L}_+}[e_{\underline{\zeta_{\mu}},-\mu}]\leq e^{(c^*_-(\underline{\mathcal{L}_{+}})+\varepsilon)\mu}e_{\underline{\zeta_{\mu}},-\mu}$, and hence, $$\underline{\mathcal{L}_+}^n[e_{\underline{\zeta_{\mu}},-\mu}](\theta,x)\leq e^{[x+n(c^*_-(\underline{\mathcal{L}_{+}})+\varepsilon)]\mu}\underline{\zeta_{\mu}}(\theta), \,  \forall
		 (n,\theta,x)\in  \mathbb{N} \times M\times \mathbb{R}.
		$$
	\end{itemize}
\end{lemma}

\begin{proof}
(i) By the definition of $c^*_-(\underline{\mathcal{L}_{+}})$, there exists $\mu:=\mu_\varepsilon >0$ such that $\underline{\lambda}(\mu)<e^{(c^*_-(\underline{\mathcal{L}_{+}})+\varepsilon)\mu}<1$.

(ii) By the definitions of $\underline{\mathcal{L}_{+,\mu}}$ and $\underline{\zeta_{\mu}}$, we have $\underline{\mathcal{L}_{+,\mu}}[\underline{\zeta_{\mu}}]=\underline{\lambda}(\mu)\underline{\zeta_{\mu}}$, and hence, $\underline{\mathcal{L}_+}[e_{\underline{\zeta_{\mu}},-\mu}](\cdot,0)=\underline{\lambda}(\mu)\underline{\zeta_{\mu}}$.  In view of the conditions of $\underline{\mathcal{L}_+}$ and the definition of $e_{\underline{\zeta_{\mu}},-\mu}$,  we see that
$$
\underline{\mathcal{L}_+}[e_{\underline{\zeta_{\mu}},-\mu}]=\underline{\lambda}(\mu)e_{\underline{\zeta_{\mu}},-\mu}\leq e^{(c^*_-(\underline{\mathcal{L}_{+}})+\varepsilon)\mu} e_{\underline{\zeta_{\mu}},-\mu}.
$$
This completes the proof.
\end{proof}

\begin{thm} \label{thm2.4}
Assume that $\underline{\mathcal{L}_+}$ satisfies all conditions in Lemma~\ref{lemma2.1} and  $\underline{\mathcal{L}_+}[\varphi]\geq Q[\varphi]$ for all $\varphi\in C_{\phi_1^*}$.
Let $\mu_\varepsilon$ be defined as in Lemma~\ref{lemma2.1}.
 If $W\in  C_{\phi_1^*}$ with $Q[W]=W$ and $\limsup\limits_{x\to -\infty}\Big[||W(\cdot,x)||e^{-\mu_\varepsilon x}\Big]<\infty$ for some  $\varepsilon\in\left(0,-c^*_-(\underline{\mathcal{L}_{+}})\right)$, then $W=0$.
In other words, $Q$ has no  nontrivial  fixed point $W$ in $C_{\phi_1^*}$ with
 $\limsup\limits_{x\to -\infty}\Big[||W(\cdot,x)||e^{-\mu_\varepsilon x}\Big]<\infty$ for some  $\varepsilon\in\left(0,-c^*_-(\underline{\mathcal{L}_{+}})\right)$.
\end{thm}

\begin{proof} Since  $\limsup\limits_{x\to -\infty}\max\limits_{\theta\in M}\Big[W(\theta,x)e^{-\mu_\varepsilon x}\Big]<\infty$, there exists $A_0>0$ such that  $W(\theta,x)\leq A_0 e^{\mu_\varepsilon x}\underline{\zeta_{\mu_\varepsilon}}(\theta)$ for all $(\theta,x)\in M\times (-\infty,-A_0]$. This, together with $\mu_\varepsilon>0$ and the boundedness of $W$, implies that $W(\theta,x)\leq A_1 e^{\mu_\varepsilon x}\underline{\zeta_{\mu}}(\theta)$ for all $(\theta,x)\in M\times \mathbb{R}$ with $A_1=\max\{A_0,||W||_{L^\infty(M\times \mathbb{R},\mathbb{R})}e^{\mu_\varepsilon A_0}\}$.
It follows from Lemma~\ref{lemma2.1} that
$$W(\theta,x)=Q^n[W](\theta,x)\leq \underline{\mathcal{L}_+}^n[W](\theta,x)\leq \underline{\mathcal{L}_+}^n[A_1e_{\underline{\zeta_{\mu_\varepsilon}},-\mu_\varepsilon}](\theta,x)\leq A_1e^{[x+n(c^*_-(\underline{\mathcal{L}_{+}})+\varepsilon)]\mu_\varepsilon}\underline{\zeta_{\mu_\varepsilon}}(\theta)$$
for all $(\theta,x)\in M\times \mathbb{R}$ and $n\in \mathbb{N}$, and hence, $W=0$.
\end{proof}

It is worthy to point out that when $c_-^*<0$, generally speaking, it is uncertain whether there is a nontrivial fixed point, and the existence of nontrivial fixed points may depend on both intermediate and limiting systems.
The example in Proposition~\ref{prop4.4} shows that even if $c_-^*<0$ and $Q$ satisfies  {\bf (UAA)} at $-\infty$, the intermediate system can give rise to nontrivial fixed points.

To present our last result in this section,  we need the following assumption on the  global asymptotic stability of a positive fixed point.
\begin{enumerate}
\item [{\bf(GAS)}] There exists $W\in C_+$ such that $\lim\limits_{n\to \infty}Q^n[\varphi]= W$ in $C$  for all $\varphi\in C_{+}\setminus\{0\}$.
\end{enumerate}

\begin{thm} \label{thm2.5}
Let {\bf (GAS)} hold for the map $Q$.  If  {\bf (SP)}, {\bf (UC)}, {\bf (GUAA$_-$)}  hold and  $\min\{c_-^*,c_+^*\}>0$, then for any
$\varphi\in  C_{+}\setminus\{0\}$, there holds
$$
\lim\limits_{n \rightarrow \infty}
\Big[\sup\{||Q^n[\varphi](\cdot, x)-W(\cdot,x)||:  x\leq n(c_+^*-\varepsilon)\}\Big]= 0
$$
for all $\varepsilon\in (0,c_+^*)$.
\end{thm}

\noindent
\begin{proof}  Given any $\gamma>0$ and $\varepsilon\in (0,c_+^*)$.  In view of Theorem~\ref{thm2.1}-(i),  there exists $\alpha_0>0$ such that
$||Q^n[\varphi](\cdot, x)-r^*\cdot {\bf {1}}_{\mathbb{R}_+}(x)||<\frac{\gamma}{3}$ and $||W(\cdot, x)-r^*\cdot {\bf {1}}_{\mathbb{R}_+}(x)||<\frac{\gamma}{3}$ for all
$(n,x)\in \mathcal{B}_{\alpha_0,\varepsilon}^+$. Thus, $||Q^n[\varphi](\cdot, x)-W(\cdot, x)||<\frac{2\gamma}{3}$ for all
$(n,x)\in \mathcal{B}_{\alpha_0,\varepsilon}^+$. It follows from {\bf (GAS)} that there exists $N_0>0$ such that $||Q^n[\varphi](\cdot, x)-W(\cdot, x)||<\frac{\gamma}{3}$ for all $x\in [-\alpha_0,\alpha_0]$ and $n>N_0$. These, together with the choices of $\mathcal{B}_{\alpha_0,\varepsilon}^+$, $\alpha_0$, and $N_0$, imply that $||Q^n[\varphi](\cdot, x)-W(\cdot, x)||<\gamma$ for all $x\in (-\infty,n(c_+^*-\varepsilon)]$ and $n>\max\{N_0,\alpha_0\}$.   \end{proof}

\section {Continuous-time  systems } \label{3sec}

In this section, we study the upward convergence, asymptotic annihilation,
equilibria  and traveling waves for continuous-time  semiflows and nonautonomous systems, respectively,  under the  assumption of  {\bf(UC)}  and {\bf(UAA)} combination.

\subsection{Continuous-time semiflows}
A map $Q:\mathbb{R}_{+}\times{C}_+\to{C}_+$ is said to be a
continuous-time semiflow on $C_+$ if  for any given vector $r\in Int(\mathbb{R}^N_+)$,
$Q|_{\mathbb{R}_{+}\times C_r}: \mathbb{R}_{+}\times{C}_r\to{C}_+$
is continuous,  $Q_0=Id|_{{C}_+}$,  and $Q_{t}\circ
Q_s=Q_{t+s}$ for all $t,s\in\mathbb{R}_+$, where $Q_t:=
Q(t,\cdot)$ for all $t\in\mathbb{R}_+$.

\begin{thm} \label{thm3.1} Let $\{Q_t\}_{t\geq 0}$ be a continuous-time semiflow on $C_+$ and
let $\varphi\in C_{+}\setminus\{0\}$. Then the following statements are valid:
\begin{itemize}
\item [{\rm (i)}] If  $Q_{t_0}$ satisfies all assumptions of Theorem~\ref{thm2.1}-(i) for some $t_0>0$,
then
$$\lim\limits_{\alpha \rightarrow \infty}
\Big[\sup\{||Q_t[\varphi](\cdot, x)-r^*\cdot {\bf {1}}_{\mathbb{R}_+}(x)||:(t,x)\in \mathcal{C}_{\alpha,\varepsilon}^+\}\Big]= 0,
$$
where $\varepsilon\in (0,\frac{c_+^*}{t_0})$ and $\mathcal{C}_{\alpha,\varepsilon}^+=\{(t,x)\in \mathbb{R}_+\times \mathbb{R}:t\geq \alpha t_0\mbox{ and } x \in [\alpha, t(\frac{c_+^*}{t_0}-\varepsilon)] \bigcup (-\infty,-\alpha]\}$.

\item [{\rm (ii)}] If  $Q_{t_0}$ satisfies all the assumptions of Theorem~\ref{thm2.1}-(ii) for some $t_0>0$,
%where $Q_+$ are defined as in Lemma~\ref{lemm2.1}-(i) with $Q$ replaced by $Q_{t_0}$,
then
$$\lim\limits_{t \rightarrow \infty}
\Big[\sup\{||Q_t[\varphi](\cdot, x)-r^*_+\cdot {\bf {1}}_{\mathbb{R}_+}(x)||:x\in t\mathcal{C}_{\varepsilon}^+\}\Big]= 0,
$$
 where $\varepsilon\in (0,\frac{1}{2t_{0}}\min\{c_+^*,c_+^*+c_-^*\})$ and $\mathcal{C}_{\varepsilon}^+=[\max\{\varepsilon,-\frac{c_-^*}{t_{0}}+\varepsilon\}, \frac{c_+^*}{t_{0}}-\varepsilon] \bigcup (-\infty,-\varepsilon]$.
\end{itemize}
\end{thm}

\noindent
\begin{proof}
Since  $\{Q_t\}_{t\geq 0}$  is an autonomous semiflow, we may assume that  $t_0=1$ in our proof. Otherwise, we consider the autonomous semiflow $\{\hat{Q}_t\}_{t\geq 0}:=\{Q_{t_0t}\}_{t\geq 0}$ instead of
$\{Q_t\}_{t\geq 0}$.
Let $\varphi\in C_{+}\setminus\{0\}$ and $K=Q[[0,1]\times \{\varphi\}]$. Then $K\subseteq C_+\setminus \{0\}$ and  $K$ is compact and uniformly bounded in $C$.

(i) Given any $\varepsilon\in (0,c_+^*)$ and $\gamma>0$.
By applying \cite[Lemma 2.3 and Theorem 3.1-(i)]{yz2025}] to $Q_1$ with $K$ and $c=0$, we have
\begin{equation*}\label{4.2}
\lim\limits_{\alpha\rightarrow \infty}
\sup \left \{||Q_n[Q_t[\varphi]](\cdot, x)-r^*||:n\geq \alpha, x\in  \left [\alpha, n(c_+^*-\frac{\varepsilon}{3})\right ],  \,  t\in [0,1]\right \}= 0.
\end{equation*}
It follows that there is $\alpha_0>0$
such that $||Q_{t+n}[\varphi](\cdot,x)-r^*||<\gamma$ for all $n\geq \alpha_0$, $x\in  \left [\alpha_0, n(c_+^*-\frac{\varepsilon}{3})\right ]$, and $t\in [0,1]$. Hence,
$||Q_{t}[\varphi](\cdot,x)-r^*||<\gamma$ for all $t>1+\alpha_0+\frac{3c_+^*}{\varepsilon}
$ and $x\in  \left [1+\alpha_0+\frac{3c_+^*}{\varepsilon}, t(c_+^*-\varepsilon)\right ]$.
On the other hand,  by slightly adapting the proof of  Theorem~\ref{thm2.1} under {\bf(GUAA$_-$)}, we  obtain
$$\lim\limits_{\alpha\rightarrow \infty}
\Big[\sup\{
||Q_{n}[\psi](\cdot,x)||:(n,-x,\psi)\in  [\alpha,\infty)^2\times K\}\Big]=0,
$$
and hence,
$$\lim\limits_{\alpha\rightarrow \infty}\Big[\sup\{||Q_{t+n}[\varphi](\cdot,x)||:(n,-x)\in  [\alpha,\infty)^2,
 \, t\in [0,1] \}\Big]=0.
$$
It then follows that for any $\gamma>0$, there exists $\alpha_0>0$ such that
$$
||Q_{t+n}[\varphi](\cdot,x)||<\gamma, \,  \, \forall n,-x\in [\alpha_0,\infty),\,
t\in [0,1].
$$
This implies that
$$
||Q_{t}[\varphi](\cdot,x)||<\gamma, \,  \, \forall t,-x\in [2(\alpha_0+1),\infty).
$$
Now the statement   (i) follows from the arbitrariness of $\gamma$.

(ii) Given any $\varepsilon\in (0,\min\{c_+^*,\frac{c_+^*+c_-^*}{2}\})$ and $\gamma>0$.
By applying  \cite[Lemma 2.3 and Theorem 3.2-(i)]{yz2025}]  to $Q_1$ with $K$, we have
\begin{equation*}\label{4.1}
\lim\limits_{n\rightarrow \infty}
\max \left \{||Q_n[Q_t[\varphi]](\cdot, x)-r^*||:\, x\in  \left [n\max\{\frac{\varepsilon}{3},-c_-^*+\frac{\varepsilon}{3}\}, n(c_+^*-\frac{\varepsilon}{3})\right ], \, t\in [0,1]\right \}= 0.
\end{equation*}
This implies that there is $N_0\in \mathbb{N}$
such that $||Q_{t+n}[\varphi](\cdot,x)-r^*||<\gamma$ for all $n>N_0$, $x\in  \left [n\max\{\frac{\varepsilon}{3},-c_-^*+\frac{\varepsilon}{3}\}, n(c_+^*-\frac{\varepsilon}{3})\right ]$, and $t\in [0,1]$. As a result,
$||Q_{t}[\varphi](\cdot,x)-r^*||<\gamma$ for all $t>1+N_0+\frac{3c_+^*}{\varepsilon}
$ and $x\in  \left [t\max\{\varepsilon,-c_-^*+\varepsilon\}, t(c_+^*-\varepsilon)\right ]$.
By the proof of (i) under {\bf(GUAA$_-$)} and the fact that  $\{t\}\times [\varepsilon t,\infty)\subseteq  [\min\{1,\varepsilon\}t,\infty)^2$, we see that  (ii) follows from the arbitrariness of $\gamma$.
\end{proof}

We need the following uniform continuity,  as introduced in \cite{yz2020},
to prove the  asymptotic annihilation for $Q_t$  without assuming  the  unilateral  uniform annihilation.

\begin{enumerate}
\item [{\bf (SC)}]
For any $t_0>0$ and $\phi\in Int(Y_+)$,
$\lim\limits_{C_{\phi}\ni\varphi\to 0}T_{-y}\circ Q_t \circ T_{y}[\varphi](\cdot,0)=0$ in $Y$ uniformly for all $(t,y)\in [0,t_0] \times \mathbb{R}$.

%For any $t_0>0$ and $\phi_k\in C_+$ with $\lim\limits_{k\to \infty}\phi_k=0$ and $\sup\limits_{k\in \mathbb{N}}||\phi_k||_{L^\infty(M\times \mathbb{R}, \mathbb{R})}<\infty$, there holds $\lim\limits_{k\to \infty}T_{-y}\circ Q_t \circ T_{y}[\phi_k]=0$ in $C$ uniformly for $ (t,y)\in [0,t_0] \times \mathbb{R}$.
\end{enumerate}

\begin{thm} \label{thm3.2}
Assume that $\{Q_t\}_{t\geq 0}$ is a continuous-time semiflow on $C_+$ such that {\bf (SC)} holds. Let $\varphi\in { C_{+}}$ and $t_0>0$.
 Then the following statements are valid:
\begin{itemize}
\item [{\rm (i)}] If $Q_{t_0}$ satisfies all the assumptions of Theorem~\ref{thm2.2}-(i) and $\varphi(\cdot,x) $ is zero for all sufficiently positive $x$, then $$\lim\limits_{t\rightarrow \infty}
\Big[\sup\{||Q_t[\varphi](\cdot,x)||: x\geq  t(\frac{\liminf\limits_{k\to \infty}c^*_{L_k}}{t_0}+\varepsilon) \mbox{ or } x\leq - t\varepsilon\}\Big]=0, \forall \varepsilon>0.$$

\item [{\rm (ii)}] If $Q_{t_0}$ satisfies all the assumptions of Theorem~\ref{thm2.2}-(ii) and $\varphi(\cdot,x) $ is zero for all sufficiently positive $x$, then $\lim\limits_{\alpha\rightarrow \infty}\Big[\sup\{||Q_t[\varphi](\cdot,x)||:|x|,t\geq \alpha \}\Big]=0$ and $$\lim\limits_{t\rightarrow \infty}\Big[\sup\{
||Q_t[\varphi](\cdot,x)||:|x|\geq t\varepsilon \}\Big]=0, \forall \varepsilon>0.$$

\end{itemize}
\end{thm}

\noindent
\begin{proof}  We only prove (i) since we may use similar arguments to deal with (ii). Since $\{Q_t\}_{t\geq 0}$  is an autonomous semiflow, we assume that  $t_0=1$ without loss of generality. Fix $\varepsilon>0$, $\gamma>0$, and $\varphi\in C_+$ with  $\varphi(\cdot,x) $ being zero for all sufficiently positive $x$.

  In view of  {\bf (UB)} and {\bf (GLC)}, we then have $\varphi\leq \phi_{k_0}^*$, $Q_n[C_{\phi_{k_0}^*}]\subseteq C_{\phi_{k_0}^*}$, and $|c^*_{L_{k_0}}-\liminf\limits_{k\to \infty}c^*_{L_k}|<\frac{\varepsilon}{6}$  for some $k_0\in \mathbb{N}$.
 It follows from  {\bf (SC)} that
there
exist $\delta=\delta(\gamma)>0$ and $d=d(\gamma)>0$ such
that if $\psi\in C_{\phi_{k_0}^*}$ with $||\psi(\cdot,x)||<\delta$ for all
$x\in [-d,d] $,  then
$||T_{-y}\circ Q_{t}\circ T_{y}[\psi](\cdot,0)||<\gamma$ for all $t\in[0,1]$ and $y\in \mathbb{R}_+$.
By virtue of  Theorem 3.3 in~\cite{yz2025},  we obtain
\begin{equation}\label{4.3}
\lim\limits_{n\rightarrow \infty}
\Big[\sup\{
||Q_n[\varphi](\cdot,x)||:x\geq n(c^*_{L_{k_0}}+\frac{\varepsilon}{3})\}\Big]=0.
\end{equation}
It follows from~\eqref{4.3} that there is an integer $n_1>0$
such that
$$
||Q_{n}[\varphi] (\cdot,x)||<\delta,  \,  \,  \forall
x\geq n(c^*_{L_{k_0}}+\frac{\varepsilon}{3}), \, n>n_1.
$$
Let $n_2=\max\{n_1,\frac{3d}{\varepsilon}\}$. Then for any $n>n_2$ and  $y\geq n(c^*_{L_{k_0}}+\frac{2\varepsilon}{3})$, we have
$$
||T_{-y}\circ Q_{n}[\varphi](\cdot,x)||<\delta, \,  \,  \forall x\in [-d,d].
$$
According to the previous discussions,  we know that for any $n>n_2$, $y\geq n(c^*_{L_{k_0}}+\frac{2\varepsilon}{3})$, and $t\in [0,1]$, there holds
\begin{eqnarray*}
||Q_{t+n}[\varphi](\cdot,y)||&=&||T_{-y} \circ Q_{t}\circ T_{y} \circ T_{-y} \circ Q_{n}[\varphi](\cdot,0)||\\
&=&||T_{-y} \circ Q_{t}\circ T_{y} [ T_{-y} \circ Q_{n}[\varphi]](
\cdot,0)||<\gamma.
\end{eqnarray*}
 In particular,
$||Q_{t}[\varphi](\cdot,y)||<\gamma$ for all $t>1+n_2
$ and $y\geq t(\liminf\limits_{k\to \infty}c^*_{L_k}+\varepsilon)$.

By using  the proof of Theorem~\ref{thm3.1}-(i) under {\bf(GUAA$_-$)} and the fact that  $\{t\}\times [\varepsilon t,\infty)\subseteq  [\min\{1,\varepsilon\}t,\infty)^2$, we see  that (i) follows from the arbitrariness of $\gamma$.
\end{proof}

The subsequent result is on  the existence of equilibrium points
under the  {\bf(UC)} and {\bf(UAA)} combination.

\begin{thm} \label{thm3.3}
Suppose that  ${Q}$ and $\underline{Q}$ are two
 monotone continuous-time semiflows on $C_+$ such that $Q_t[C_{r^{**}}]\subseteq C_{r^{**}}$, $Q_t[\varphi]\geq \underline{Q}_t[\varphi]\geq \underline{Q}_t[\psi]$ for all $(t,\varphi,\psi)\in \mathbb{R}_+\times C_{r^{**}}$ with $\varphi\geq \psi$, and for some   $t_0>0$,
  $\underline{Q}_{t_0}$ satisfies all conditions in Theorem~\ref{thm2.3} with  parameters $r$ and $c$ being underlined. If  $\{Q_t[r^{**}]\}_{t\in\mathbb{R}_+}$ is precompact in $C_+$, then  $Q$ has a  nontrivial  equilibrium point $W$ in $C_{r^{**}}$ such that $W(\cdot,-\infty)=0$ and
$W(\cdot,\infty)=r^*$.
\end{thm}

\noindent
\begin{proof}   By applying Theorem~\ref{thm2.3}  to $\underline{Q}_{t_0}$, we obtain that $\underline{Q}_{t_0}$ has a  nontrivial  fixed point $\underline{W}$ in $C_{r^{**}}$ such that $\underline{W}(\cdot,\infty)=\underline{r}^*$. According to the compactness and monotonicity of $\{Q_t[r^{**}]\}_{t\in\mathbb{R}_+}$, we know that $W=\lim\limits_{t\to\infty}Q_t[r^{**}]$
for some $W\in C_{r^{**}}$.  Note that $W\geq \underline{W}$ due to $Q_t[r^{**}]\geq \underline{Q}_t[r^{**}]$.  Hence,
$W(\cdot,\infty)=r^*$ follows from \cite[Lemma 3.4]{yz2025}. According to the proof of Theorem~\ref{thm3.1}-(i) under {\bf(GUAA$_-$)}, we then have $W(\cdot,-\infty)=0$.
\end{proof}

%\begin{rem}  \label{rem3.2-sp-monoty}   As a consequence of  Remark~\ref{rem3.1-sp-monoty} and the proof of Theorem~\ref{thm3.3},
%in Theorem~\ref{thm3.3}-(i), we can use
%$\min\{c_+^*,c_-^*\}>0$ to replace  the assumption that  $T_{-y}\circ Q_l[\varphi]\geq Q_l\circ T_{-y}[\varphi]$ for all $l\in \mathbb{N}$ and $(y,\varphi)\in \mathbb{R}_+\times C_{r^*_l}%$.\end{rem}

By the  contradiction argument  and  Theorem~\ref{thm2.4}, as applied  to $Q_{t_0}$, we easily obtain  the following result about the non-existence of the nontrivial equilibrium points.

\begin{thm} \label{thm3.4}
Assume that $Q$ is a continuous-time semiflow on $C_+$ and there exist $t_0>0$ and $\underline{\mathcal{L}_+}:C_+\to C_+$  such that $(Q_{t_0},\underline{\mathcal{L}_+})$ satisfies all the conditions of Theorem~\ref{thm2.4}. Then $Q$ has no  nontrivial equilibrium point $W$ in $C_{r^{**}}$  with
$\lim\limits_{x\to -\infty}\Big[||W(\cdot,x)||e^{-\mu_\varepsilon x}\Big]<\infty$ for some  $\varepsilon\in\left(0,-c^*_-(\underline{\mathcal{L}_{+}})\right)$. Here $\mu_\varepsilon$ is defined as in Lemma~\ref{lemma2.1}.
\end{thm}

To present our last result in this subsection, we need the following assumption.

\

 {\bf(GAS-CSF)} There exists $W\in C_+ $ such that $\lim\limits_{t\to \infty }Q_t[\varphi]=W$ in $C$ for all $\varphi\in C_+\setminus \{0\}$.

\begin{thm} \label{thm2.10}
Assume that $\{Q_t\}_{t\geq 0}$ is a continuous-time semiflow on $C_+$ such that $\{Q_t\}_{t\geq 0}$ satisfies {\bf (GAS-CSF)}. Let $\varphi\in C_{+}\setminus\{0\}$ and $t_0\in (0,\infty)$.  If $Q_{t_0}$ satisfies all the assumptions of Theorem~\ref{thm2.5} with $Q$ replaced by $Q_{t_0}$, then $$\lim\limits_{t\rightarrow \infty}
\Big[\sup\{||Q_t[\varphi](\cdot, x)-W(\cdot,x)||:  x\leq t(\frac{c_+^*}{t_0}-\varepsilon)\}\Big]= 0$$ for all  $\varepsilon\in (0,\frac{c_+^*}{t_0})$.

\end{thm}

\noindent
\begin{proof}
Given any $\gamma>0$ and $\varepsilon\in (0,\infty)$.  Applying Theorem~\ref{thm3.1} to $Q_t$,  we see that there exists $\alpha_0>0$ such that
$||Q_t[\varphi](\cdot, x)||<\frac{\gamma}{3}$ and $||W(\cdot, x)||<\frac{\gamma}{3}$ for all
$t\in [\alpha_0,\infty), x\in (-\infty,-\alpha_{0}]\bigcup [\alpha_{0},t(\frac{c_+^*}{t_0}-\varepsilon)]$. Thus, we have
$$
||Q_t[\varphi](\cdot, x)-W(\cdot, x)||<\frac{2\gamma}{3},\, \,  \forall
t\geq \alpha_0,  \, x\in (-\infty,-\alpha_{0}]\bigcup [\alpha_{0},t(\frac{c_+^*}{t_0}-\varepsilon)].
$$
 It follows from {\bf (GAS-CSF)} that there exists $T_0>0$ such that
 $$
 ||Q_t[\varphi](\cdot, x)-W(\cdot, x)||<\frac{\gamma}{3}, \, \,  \forall x\in [-\alpha_0,\alpha_0], \,  t>T_0.
 $$
 These, together with the choices of $\alpha_0$ and $T_0$, imply that $||Q_t[\varphi](\cdot, x)-W(\cdot, x)||<\gamma$ for all  $t>\max\{T_0,\alpha_0\}$ and $x\leq t(\frac{c_+^*}{t_0}-\varepsilon)$. Now  the arbitrariness of $\gamma$ gives rise to our results.  \end{proof}

\subsection{Nonautonomous systems}
 Let
 $P:\mathbb{R}_{+}\times  {C}_+\to{C}_+$  be a map such that for any given $r\in Int(\mathbb{R}^N_+)$,
$P|_{\mathbb{R}_{+}\times C_r}: \mathbb{R}_{+}\times {C}_r\to{C}_+$
is continuous.  For any given $c\in \mathbb{R}$, we define a family of
mappings $Q_t:=T_{-ct}\circ P[t,\cdot]$  with parameter
$t\in \mathbb{R}_+$.

Under the assumption of  {\bf(UC)} and {\bf(GUAA)} combination,
we can apply Theorem~\ref{thm3.1} to  $\{\varphi,\{Q_t\}_{t\in \mathbb{R}_+}\}$ to obtain the following result.

\begin{thm} \label{thm3.6}  Assume  that there exists   $c\in \mathbb{R}$ such that $Q_t:=T_{-ct}\circ P[t,\cdot]$ is a continuous-time semiflow on $C_+$ and
let $\varphi\in C_{+}\setminus\{0\}$. Then the following statements are valid:
\begin{itemize}
\item [{\rm (i)}] If  $Q_{t_0}$ satisfies all the assumptions of Theorem~\ref{thm2.1}-(i) for some $t_0>0$, %with $Q$ replaced by $Q_{t_0}$,
then $$\lim\limits_{\alpha \rightarrow \infty}
\Big[\sup\{||P[t,\varphi](\cdot, x)-r^*\cdot {\bf {1}}_{\mathbb{R}_+}(x-ct)||:(t,x)\in \mathcal{D}_{\alpha,\varepsilon,c}^+\}\Big]= 0,$$ where $\varepsilon\in (0,\frac{c_+^*}{t_0})$ and $\mathcal{D}_{\alpha,\varepsilon,c}^+=\{(t,x)\in \mathbb{R}_+\times \mathbb{R}:t\geq \alpha t_0\mbox{ and } x \in [\alpha+ct, t(c+\frac{c_+^*}{t_0}-\varepsilon)] \bigcup (-\infty,ct-\alpha]\}$.

\item [{\rm (ii)}] If  $Q_{t_0}$ satisfies all the assumptions of Theorem~\ref{thm2.1}-(ii) for some $t_0>0$,% with $Q$ replaced by $Q_{t_0}$,
 then $$\lim\limits_{t \rightarrow \infty}
\Big[\sup\{||P[t,\varphi](\cdot, x)-r^*_+\cdot {\bf {1}}_{\mathbb{R}_+}(x-ct)||:x\in t\mathcal{D}_{\varepsilon,c}^+\}\Big]= 0,$$ where $\varepsilon\in (0,\frac{1}{2t_{0}}\min\{c_+^*,c_+^*+c_-^*\})$ and $\mathcal{D}_{\varepsilon,c}^-=[\max\{\varepsilon,c-\frac{c_-^*}{t_{0}}+\varepsilon\}, c+\frac{c_+^*}{t_{0}}-\varepsilon] \bigcup (-\infty,c-\varepsilon]$.
\end{itemize}\end{thm}

By  Theorems~\ref{thm3.2}, as applied  to $\{Q, \varphi\}$, we have the following result
on the  asymptotic annihilation.

\begin{thm} \label{thm3.7}
Assume that $c\in \mathbb{R}$ and $Q_t:=T_{-ct}\circ P[t,\cdot]$ is a continuous-time semiflow on $C_+$ such that {\bf (SC)} holds, and let $\varphi\in { C_{+}}$.
 Then the following statements are valid:
\begin{itemize}
\item [{\rm (i)}] If $Q_{t_0}$ satisfies all the assumptions of Theorem~\ref{thm3.2}-(i)
for some  $t_0>0$  and $\varphi(\cdot,x) $ is zero for all sufficiently positive $x$, then $$\lim\limits_{t\rightarrow \infty}
\Big[\sup\{
||P[t,\varphi](\cdot,x)||: x\geq  t(c+\frac{\liminf\limits_{k\to \infty}c^*_{L_k}}{t_0}+\varepsilon) \mbox{ or } x\leq  t(c-\varepsilon)\}\Big]=0, \forall \varepsilon>0.$$

\item [{\rm (ii)}] If $Q_{t_0}$ satisfies all the assumptions of Theorem~\ref{thm3.2}-(ii)
for some  $t_0>0$ and $\varphi(\cdot,x) $ is zero for all sufficiently positive $x$, then $\lim\limits_{\alpha\rightarrow \infty}\Big[\sup\{||P[t,\varphi](\cdot,x)||:|x-ct|,t\geq \alpha \}\Big]=0$ and $$\lim\limits_{t\rightarrow \infty}\Big[\sup\{
||P[t,\varphi](\cdot,x)||:|x-ct|\geq t\varepsilon \}\Big]=0, \forall \varepsilon>0.$$

\end{itemize}
\end{thm}

We say that $W(\cdot,x-ct)$ is a travelling wave of $P$ if
$W:M\times\mathbb{R}\to \mathbb{R}_+$ is a bounded and
nonconstant continuous function such that
$P[t,W](\theta,x)=W(\theta,x-tc)$ for all $(\theta,x)\in
M\times \mathbb{R}$ and $t\in \mathbb{R}_+$, and that
$W$ connects $0$ to $r^*$ if $W(\cdot,-\infty)=0$ and
$W(\cdot,\infty)=r^*$.
As a consequence  of Theorem \ref{thm3.3},
we have the following result on the existence of travelling waves.

\begin{thm} \label{thm3.8} Suppose that $t_0,t_1>0$, $c\in \mathbb{R}$, and $Q_t:=T_{-ct}\circ P[t,\cdot]$, $\underline{Q}_t$ are continuous-time semiflows on $C_+$ such that $\{Q_t,\underline{Q}_{t}\}$ satisfies all conditions in Theorem~\ref{thm3.3}. If  $\underline{Q}_{t_0}$ satisfies all conditions in Theorem~\ref{thm2.3} with parameters $r$ and $c$ being underlined, then  $P[t,\cdot]$ has a  nontrivial  travelling wave $W(\cdot,x-ct)$ in $C_{r^{**}}$ such that $W(\cdot,-\infty)=0$ and
$W(\cdot,\infty)=r^*$.
\end{thm}

Regarding the global attractivity of the positive travelling wave, we can use
Theorem~\ref{thm2.10} to establish  the following result.
\begin{thm} \label{thm3.9}
Assume that  there exists   $c\in \mathbb{R}$ such that $Q_t:=T_{-ct}\circ P[t,\cdot]$  is a continuous-time semiflow on $C_+$ and  $\{Q_t\}_{t\geq 0}$ satisfies {\bf (GAS-CSF)}. Let $\varphi\in C_{+}\setminus\{0\}$ and $t_0\in (0,\infty)$.
 If $Q_{t_0}$ satisfies all the assumptions of Theorem~\ref{thm2.5} with $Q$ replaced by $Q_{t_0}$, then $$\lim\limits_{t\rightarrow \infty}
\Big[\sup\{||P[t,\varphi](\cdot, x)-W(\cdot,x-ct)||: x\leq t(c+\frac{c_+^*}{t_0}-\varepsilon)\}\Big]= 0,$$ where $\varepsilon\in (0,\frac{c_+^*}{t_0})$.
\end{thm}

Following  \cite[Section 3.1]{Zhaobook}, we say a map $Q:\mathbb{R}_{+}\times{C}_+\to{C}_+$ is a
 continuous-time  $\omega$-periodic semiflow on $C_+$ if  for any given  $r\in Int(\mathbb{R}_+^N)$,
 $Q|_{\mathbb{R}_{+}\times C_r}: \mathbb{R}_{+}\times{C}_r\to{C}_+$
 is continuous,  $Q_0=Id|_{{C}_+}$,  and $Q_{t}\circ
 Q_{\omega}=Q_{t+\omega}$ for  some number $\omega>0$ and all $t\in\mathbb{R}_+$,  where $Q_t:=
 Q(t,\cdot)$ for all $t\in\mathbb{R}_+$.

\begin{rem}
In the case where $Q_t:=T_{-ct}\circ P[t,\cdot]$ is a continuous-time
$\omega$-periodic semiflow on $C_+$, we can apply the main results in Section~\ref{3sec} to the Poincar\'e map $Q_{\omega}$  to
establish the spreading properties and the forced time-periodic traveling waves with speed $c$ for the nonautonomous evolution system
$P[t,\cdot]$. We refer to\cite{LYZ} for the Poincar\'e map approach to
monotone periodic semiflows.
\end{rem}

\section{ An integro-difference equation}
In this section, we apply our developed theory  in Section 2 to a  class of spatially heterogeneous integro-difference equations under the assumption of GB combination for two limiting
systems.  We also provide  an example to show that even if $c_-^*<0$ and $Q$ satisfies  {\bf (UAA)} at $-\infty$, the intermediate system can give rise to nontrivial fixed points.

In the below, we always assume that $f$ and $f_\pm$ satisfy the following conditions:
\begin{enumerate}
	\item[(B1)] $f\in C(\mathbb{R}\times\mathbb{R}_{+},\mathbb{R}_+), f(s,\cdot)\in C^{1}(\mathbb{R}_{+},\mathbb{R}), f(s,0)=0 \mbox{  for all }
	s\in \mathbb{R}$, $f(\mathbb{R}\times (0,\infty))\subseteq (0,\infty)$, $ f(s,u)$   is nondecreasing   in  $u\in (0,\infty)$;
	\item[(B2)]  $\lim\limits_{s\to \pm \infty}f(s,\cdot)=f_\pm(\cdot)$  in   $C_{loc}^{1}(\mathbb{R}_{+},\mathbb{R})$, $\frac{{\rm d} f_+(0)}{\rm d u}>1>\frac{{\rm d} f_-(0)}{\rm d u}$, $0\leq f_\pm(u)\leq \frac{{\rm d} f_\pm(0)}{\rm d u}u$ for all $u\in \mathbb{R}_+$, $\mbox{ and  } \{u>0:f_+(u)=u\}=\{u_+^*\} \mbox{ for some } u^{*}_+ >0$;
	\item[(B3)] There exists a sequence $\{u_k^*\}_{k\in \mathbb{N}}$ in $(0,\infty)$ such that $u_k^*<u_{k+1}^*$, $\lim\limits_{k\to\infty}u_k^*=\infty$, and $f(\mathbb{R}\times [0,u_k^*])\subseteq  [0,u_k^*]$ for all $k\in \mathbb{N}$.
\end{enumerate}

 Here we do not assume  that  $f(s,u)\leq \partial_u f(s,0)u$ for   all $(s,u)\in \mathbb{R}\times \mathbb{R}_+$.  We can easily verify the following  elementary result, which will be used  in our study of  the non-existence of fixed points, steady states and traveling waves.

\begin{lemma}\label{lem4.1} Assume that   one of the following conditions holds true:
	\begin{itemize}
		\item [{\rm (a)}]  $f(s,u)\leq \partial_u f(s,0) u$ for all $(s,u)\in \mathbb{R}\times \mathbb{R}_+$;
		
		\item [{\rm (b)}] $\liminf\limits_{(s,u)\to(-\infty,0)}\Big[\partial_u f(s,0)-\partial_u f(s,u)\Big]\geq 0$ and $ \liminf\limits_{s\to-\infty}\Big[\min \{\partial_u f(s,0)u-f(s,u):u\in [0,u_l^*]\}\Big]\geq 0$ for all $l\in\mathbb{N}$.
		
	\end{itemize}
	Then  for any $u^{**}>0$ and $\gamma>0$, there exist $\underline{R}^+=\underline{R}^+_{\gamma,u^{**}}\in C(\mathbb{R},\mathbb{R}_+)$ such that
	\begin{itemize}
		\item[{\rm (i)}]  $\underline{R}^+(-\infty)=\gamma+\limsup\limits_{s\to - \infty}\partial_u f(s,0)$ and $\underline{R}^+( \infty)=\gamma+\frac{{\rm d}f_+(0)}{{\rm d}u}$;

		\item[{\rm (ii)}] $f(s,u)\leq \underline{R}^+(s)u$ for all $(s,u)\in \mathbb{R}\times[0,u^{**}]$;

		\item[{\rm (iii)}] $\underline{R}^+(s)\leq \gamma+\frac{{\rm d} f_+(0)}{{\rm d}  u}$ for all $(s,u)\in \mathbb{R}\times[0,u^{**}]$ provided that $f(s,u)\leq \frac{{\rm d} f_+(0)}{{\rm d}  u}u$ for all $(s,u)\in \mathbb{R}\times[0,u^{**}]$.
	\end{itemize}
\end{lemma}

Consider the integro-difference equation
		\begin {equation} \left\{
		\begin{array}{rcl}
			u_{n}(x) & = & \int_{\mathbb{R}}f(y,u_{n-1}(y))\mathfrak{k}(x-y){\rm{d}} y,
			\\
			u_{0}&\in & BC(\mathbb{R},\mathbb{R}_+),
		\end{array} \right.
		\label{6.1-1}
		\end {equation}
		where $\mathfrak{k}:\mathbb{R}\rightarrow \mathbb{R}$ is a nonnegative
		continuous function satisfying  $\int_{\mathbb{R}}\mathfrak{k}(y){\rm{d}} y=1$. As in \cite{hz2008,lbbf2016,lfm2015}, we assume that $\int_{\mathbb{R}}e^{\mu_\pm
			y}\mathfrak{k}(y){\rm{d}} y<\infty$ for some  $(\mu_+,\mu_-)\in (0,\infty)\times (-\infty,0)$.
		
		Let $M=\{0\}$, $C=BC(\mathbb{R},\mathbb{R})$, $C_+=BC(\mathbb{R},\mathbb{R}_+)$, and $C_\phi=C_+\cap (\phi-C_+)$ for any $\phi\in C_+$.
		Clearly, we can  identify $\mathbb{R}$ and $M\times \mathbb{R}$  to apply the developed theory
		 directly  to system \eqref{6.1-1}.  Let
		\begin {equation}
		c^{*}_{\pm}(\infty):=\inf\limits_{\mu>0}\frac{\ln\Big[\frac{{\rm d} f_+(0)}{{\rm d} u}\int_{\mathbb{R}}e^{\pm\mu
				y}\mathfrak{k}(y){\rm{d}} y\Big]}{\mu}.
		\label{6.1-2}
		\end {equation}
		Define
		$$
		Q[\phi;f(\cdot,\cdot)](x)=\int_{\mathbb{R}}f(x-y,\phi(x-y))\mathfrak{k}(y){\rm{d}} y,\,\,\, \forall \phi\in
		C_+,\,\,x\in \mathbb{R},
		$$
		and
		$$
		Q_\pm[\phi;f_\pm(\cdot)](x):=Q[\phi;f_\pm(\cdot)](x)=\int_{\mathbb{R}}f_\pm(\phi(x-y))\mathfrak{k}(y){\rm{d}} y,\,\,\, \forall \phi\in
		C_+,\,\,x\in \mathbb{R}.
		$$
		
			\begin{prop} \label{prop4.3} The following statements are valid:
			\begin{itemize}
				\item [{\rm (i)}]  $Q[C_{u_k^*};f(\cdot,\cdot)]\subseteq C_{u_k^*}$, and $Q[C_{u_k^*};f(\cdot,\cdot)]$ is precompact in $C$ for all $k \in \mathbb{N}$.
				
				\item [{\rm (ii)}]  If $\phi\in C_+$ and there exist two intervals $[a,b]$ and $[c,d]$ in $\mathbb{R}$ such that  $\phi([a,b]), f([a,b]\times (0,\infty)), \mathfrak{k}([c,d])\subseteq (0,\infty)$, then  $Q[\phi;f(\cdot,\cdot)](x)>0$ for all $x\in [a+c,b+d]$ .
				Hence, {\bf {(SP)}} and {\bf {(SP$_-$)}} hold provided that $f(\mathbb{R}\times (0,\infty))\subseteq (0,\infty)$, $\mathfrak{k}(0,\infty)\neq\{0\}$, and
				$\mathfrak{k}(-\infty,0)\neq\{0\}$.
				
				\item [{\rm (iii)}]  If $\lim\limits_{s\to  \pm \infty}f(s,\cdot)= f_\pm$ in $L^\infty_{loc}(\mathbb{R}_+,\mathbb{R})$, then $\lim\limits_{y\to  \pm \infty}Q^n[T_{y}[\phi];f(\cdot,\cdot)](\cdot,\cdot+y)= (Q_\pm)^n[\phi;f_\pm(\cdot)]$ in $C$ for all $n\in \mathbb{N}$. Moreover, $\lim\limits_{(k,y)\to  (\infty, \infty)}Q^n[T_{\pm y}[\phi_k];f(\cdot,\cdot)](\cdot,\cdot\pm y)= (Q_\pm)^n[\phi;f_\pm(\cdot)]$ in $C$  for all $n\in \mathbb{N}$ provided that $\{\varphi_k\}_{k\in \mathbb{N}\cup\{0\}}\in C_r$ for some  $r>0$ and $\lim\limits_{k\to \infty}||\varphi_k-\varphi_0||=0$.
					\item [{\rm (iv)}]  If $f(s,u)$ is nondecreasing in $u\in \mathbb{R}_+$ for all $s\in \mathbb{R}$, then  $Q[\phi;f(\cdot,\cdot)]\geq Q[\psi;f(\cdot,\cdot)]$ for all $\phi,\psi\in C_+$ with $\phi\geq \psi$.
				
				\item [{\rm (v)}]  If $f(s,u)$ is nondecreasing in $s\in \mathbb{R}$ for all $u\in \mathbb{R}_+$, then  $T_{-y}[Q[T_y[\phi];f(\cdot,\cdot)]]\geq Q[\phi;f(\cdot,\cdot)]$ for all $(\phi,y)\in C_+\times \mathbb{R}_+$.
				
				\item [{\rm (vi)}]  If $f(s,\alpha u)\geq \alpha f(s,u)$ for all $(s,u,\alpha)\in \mathbb{R}\times \mathbb{R}_+\times [0,1]$, then  $Q[\alpha \phi;f(\cdot,\cdot)]\geq \alpha Q[\phi;f(\cdot,\cdot)]$ for all $(\phi,\alpha)\in C_+\times[0,1]$.
			\end{itemize}
			\end{prop}

\begin{prop} \label{prop4.4} The following statements are valid:
			\begin{itemize}
				\item [{\rm (i)}] $Q[\cdot;f(\cdot,\cdot)]$ satisfies  {\bf (UC)} with $c^*_+( \infty)+c^*_-( \infty)>0$.
				$T_{y}\circ Q_{-l}[\varphi]\geq Q_{-l}\circ T_{y}[\varphi]$ for all $l\in \mathbb{N}$ and $(y,\varphi)\in \mathbb{R}_+\times C_+$ with $ly\geq 0$,  where $Q_l$ is defined as in  {\bf (ACH$_-$)} if (B1$_-$) or (B3$_-$) hold.
					
				\item [{\rm (ii)}]  $(Q[\cdot;f(\cdot,\cdot)],Q[\cdot;f_-])$ satisfies {\bf (GUAA$_-$)}
				
				\item [{\rm (iii)}]  $Q[\cdot;f(\cdot,\cdot)]$ satisfies {\bf (GLC)} with $c_+^*(\infty)$.
				
					\item [{\rm (iv)}]  If $c^*_+( \infty)\leq 0$, then there exists $\gamma\in (0,\frac{{\rm d} f_+(0)}{{\rm d} u}-1)$  such that
			 $Q[\varphi;f(\cdot,\cdot)]\geq Q[\varphi;f_{\gamma,u_1^*}(\cdot,\cdot)], \forall \varphi\in C_{u_1^*}$
			 and
			 $\inf\limits_{\mu>0}\frac{\ln\Big[(\frac{{\rm d} f_+(0)}{{\rm d} u}-\gamma)\int_{\mathbb{R}}e^{\pm\mu y}\mathfrak{k}(y){\rm{d}} y\Big]}{\mu}>0$, where $f_{\gamma,u_1^*}$ is defined as in Lemma 5.1 in~\cite{yz2025}.
		\end{itemize}
		\end{prop}
		
		\noindent
		\begin{proof} 			
			(i)  Clearly,  {\bf (UC)} follows from \cite[Theorem 2.1-(ii)]{hz2008}.
		
		(ii)
 By virtue of $ f_-(u)< u  \mbox{  for all  } u\in (0,\infty)$, we easily see that
			$$
			(f_-)^{1+n}([0,r])\subseteq (f_-)^n([0,r]) \mbox{ and }\bigcap\limits_{n\in \mathbb{N}}(f_-)^n([0,r])=\{0\},
			\quad \forall (n,r)\in \mathbb{N}\times (0,\infty).
			$$
			According to the definition of $Q[\cdot;f_-]$, we can verify that $Q[C_r;f_-]\subseteq C(\mathbb{R},f_-([0,r]))$, and hence, $Q^n[C_r;f_-]\subseteq C(\mathbb{R},(f_-)^n([0,r]))$ for all $(n,r)\in \mathbb{N}\times (0,\infty)$. This, together with Proposition~\ref{prop4.3}-(iii),  implies statement (iii).
			
			(iii) For any  $l\in \mathbb{N}$, let us denote $R_{l}=\overline{R}_{\frac{1}{l},u_l^*}$, where $\overline{R}_{\frac{1}{l},u_l^*}$ is defined as in Lemma 5.2 in~\cite{yz2025}. Then $R_{l}$ is bounded and nonincreasing with $R_{l}(-\infty)>R_{l}(\infty)=\frac{\partial f_+(0)}{{\partial } u}+\frac{1}{l}$ and $f(s,u)\leq R_{l}(s)u$ for all $(s,u)\in \mathbb{R}\times [0,u_l^*]$. 	Let
			$$
			e_{\phi,\mu}(x)=\phi(x) e^{-\mu x}, \, \, \,  \forall (\phi,\mu,x)\in C\times \mathbb{R}\times \mathbb{R}, \quad
			 \tilde{C}:= \left\{\sum\limits_{k=1}^p  e_{\phi_k,\mu_k}:(\phi_k,\mu_k,p)\in  C\times\mathbb{R}\times \mathbb{N} \right\}.
			 $$
			Define
			$$
			L_{l}[\phi](x)=\int_{ \mathbb{R}}\phi(x-y) R_{l}(x-y)
			\mathfrak{k}(y) {\rm d} y \mbox{ and } (L_{l})_+[\phi](x)=R_{l}(\infty)\int_{ \mathbb{R}}\phi(x-y)
			\mathfrak{k}(y) {\rm d} y,\, \, \,  \forall (x,\phi)\in \mathbb{R}\times \tilde{C}.
			$$
			It then easily follows that  $(L_{l},(L_{l})_+)$ satisfies {\bf(LC)}, $Q\leq L_{l}$ in $C_{u_l^*}$, and  $\lim\limits_{l\to \infty}c^*_{(L_{l})_+}=c_+^*(\infty)$ for all $l\in \mathbb{N}$.
			
			(iv) follows from $c^*_-( \infty)>-c^*_+( \infty)\geq 0$ and \cite[Lemma 5.1]{yz2025}.
			\end{proof}
		
		To simplify our restriction on initial values, in the rest of  this section we  always assume that $\mathfrak{k}(0,\infty)\neq\{0\}$,
			$\mathfrak{k}(-\infty,0)\neq\{0\}$, and  $f(\mathbb{R}\times (0, \infty))\subseteq (0,\infty)$. For example, if $f((-\infty,0) \times (0, \infty))=\{0\}$ and $\phi|_{\mathbb{R}_+}=0$, then $Q[\phi]=0$, which contradicts the assumptions {\bf(SP)} and (SP$_-$).
			However, if we replace $C_+\setminus\{0\}$ with  the set of
			functions having  (SP)/(SP$_-$)-type features, we can  slightly adapt the developed  theory to obtain the same conclusions.
			Here we should point out that the
			authors of \cite{lbbf2016,lfm2015} considered the case where  the initial values $\phi>0$ and  $f(x,u)>0$ for all $u\in (0,\infty)$.
			
		Now we are ready to present the main result for system ~\eqref{6.1-1}.

			\begin{thm}\label{thm4.1}
			  Let  $c^*_\pm( \infty)$ be defined as in ~\eqref{6.1-2}. Then the following statements are valid:
			\begin{itemize}
				\item [{\rm (i)}]
				If $\phi\in C_+\setminus \{0\}$, $c_+^*(\infty)>0$ and $\varepsilon\in (0,\frac{1}{2}\min\{c_+^*(\infty),c_+^*(\infty)+c_-^*(\infty)\})$, then $$\lim\limits_{n\rightarrow \infty}
				\Big[\sup\{|u_n^\phi( x)-u_+^*\cdot {\bf {1}}_{\mathbb{R}_+}(x)|:x\in n \mathcal{A}_{\varepsilon}^+\}\Big]= 0,$$ where $\mathcal{A}_{\varepsilon}^+=(-\infty,-\varepsilon]\bigcup [\max\{\varepsilon,-c_-^*(\infty)+\varepsilon\}, (c_+^*(\infty)-\varepsilon)]$. If,  in addition,  $c_{-}^*(\infty)>0$, then $$\lim\limits_{\alpha\rightarrow \infty}\Big[\sup\{|u_n^\phi(x)-u_+^*\cdot {\bf {1}}_{\mathbb{R}_+}(x)|:(n,x)\in \mathcal{A}_{\alpha,\varepsilon}^+\}\Big]= 0,$$ where  $\mathcal{A}_{\alpha,\varepsilon}^+=\{(n,x)\in \mathbb{N}\times \mathbb{R}:n\geq \alpha \mbox{ and } x \in [\alpha, n(c_+^*(\infty)-\varepsilon)] \bigcup (-\infty,-\alpha]\}$.
				
				\item [{\rm (ii)}]
				If $\varepsilon\in (0,\infty)$, and $\phi\in C_+$ satisfies that $\phi(\cdot,x)$ is zero for all sufficiently positive $x$, then $$\lim\limits_{n\rightarrow \infty}
				\Big[\max\{|u_n^\phi(x)|:x\geq n \max\{\varepsilon,c_+^*(\infty)+\varepsilon\} \}\Big]= 0.$$
				Moreover, if $c_{+}^*(\infty)<0$, then $\lim\limits_{\alpha\rightarrow \infty}\Big[\sup\{|u_n^\phi(x)|:n\in \mathbb{N} \mbox{ and } x\geq \alpha \}\Big]=0$.

				\item [{\rm (iii)}] If $c_-^*(\infty)>0$, then  $Q$ has a  nontrivial  fixed point $W$ in $C_+$ such that
				$W(\infty)=u_+^*$ and $W(-\infty)=0$.
				If $c_-^*(\infty)< 0$,  $\limsup\limits_{s \to -\infty}\partial_u f(s,0)<1$, either (a) or (b) in Lemma~\ref{lem4.1} holds, and $f(s,u)\leq \frac{{\rm d}  f_+(0)}{{\rm d} u}u$ for all $(s,u)\in \mathbb{R}\times \mathbb{R}_+$, then $Q$ has no  nontrivial  fixed point $W$ in $C_+$.
				%such that$W(\infty)=u_+^*$ and $W(-\infty)=0$.
				
				\item [{\rm (iv)}] Assume that  $\limsup\limits_{s \to -\infty}\partial_u f(s,0)<1$,  $f(x,\cdot)$ is  nondecreasing and strictly subhomogeneous on $(0,\infty)$ for each  $x\in (0,\infty)$. If $\min\{c_-^*(\infty),c_+^*(\infty)\}>0$, then $$\lim\limits_{n\rightarrow \infty}\Big[\sup\{|u_n^\phi(x)-W(x)|: x \leq n(c_+^*(\infty)-\varepsilon)\}\Big]= 0$$ for all $(\varepsilon,\varphi)\in (0,c_+^*(\infty))\times C_+\setminus\{0\}$, where $W$ is defined as in (iii).
				
			\end{itemize}
			
		\end{thm}
		
		\noindent
		\begin{proof}
			The former conclusion of (i) follows from Theorem~\ref{thm2.1}-(ii), as applied to
			$(Q,\varphi)$ and $(\mathcal{S}\circ Q\circ \mathcal{S},\mathcal{S}[\varphi])$, respectively. And the latter
			conclusion of (i) follows from Theorem~\ref{thm2.1}-(i).
			
			(ii) follows from Theorem \ref{thm2.2}.
			
			(iii)  If $c_-^*(\infty)>0$, then we can easily obtain the existence of $W$ with $W(\infty)=u_+^*$ and $W(-\infty)=0$ by using  Theorem~\ref{thm2.3}.
			
			Now assume that  $c_-^*(\infty)<0$, $\limsup\limits_{s \to -\infty}\partial_u f(s,0)<1$,  $f(s,u)\leq \frac{{\rm d}  f_+(0)}{{\rm d} u}u$ for all $(s,u)\in \mathbb{R}\times \mathbb{R}_+$, and  either (a) or (b) in Lemma~\ref{lem4.1} holds.  Suppose that  $Q$ has a nontrivial  fixed point $W$ in $C_+$.
			Then $W>0$ and $c_+^*(\infty)>-c_-^*(\infty)>0$.
			By the former conclusion in (i), it follows  that  $W(\cdot,\infty)=u_+^*$ and $W(\cdot,-\infty)=0$.
			Select $l_0\in \mathbb{N}$ and $\gamma_0\in \left(0,\frac{1-\limsup\limits_{s \to -\infty}\partial_u f(s,0)}{5}\right)$ with $u_{l_0}^*>||W||_{L^\infty(\mathbb{R},\mathbb{R})}$ and $c_-^*(\underline{L_+})<0$, where $$\underline{L_+}[\phi](x)=(\gamma_0+ \frac{{\rm d}f_+(0)}{{\rm d} u})\int_{\mathbb{R}}\phi(x-y)\mathfrak{k}(y){\rm d} y, \,  \, \forall (x,\phi)\in \mathbb{R}\times \tilde{C}.$$
			Let $\underline{R}:=\underline{R}^+_{\gamma_0,u_{l_0}^{*}}$ be
			given  as in Lemma~\ref{lem4.1}, and define
			 $$\underline{L}[\phi](x)=\int_{\mathbb{R}}\underline{R}(x-y)\phi(x-y)\mathfrak{k}(y){\rm d} y,   \,  \,  \forall (x,\phi)\in \mathbb{R}\times \tilde{C}.$$
			It  then follows that
			$$
			Q[\phi]\leq \underline{L}[\phi]\leq \underline{L_+}[\phi],  \,  \,  \forall
			\phi\in C_{u_{l_0}^*},
			$$
			and
			$$
			\underline{L}[\phi]\leq \underline{L_+}[\phi],  \,  \,  \forall \phi\in \tilde{C}
			 \,  \,   \text{with}  \,  \,   \phi\geq 0.
			$$
			According to Lemma~\ref{lemma2.1},  there exist $\varepsilon_0\in\left(0,-c^*_-(\underline{L_{+}})\right)$ and $\mu_0\in (0,\infty)$
			such that
			%$\lambda(-\mu_0)<e^{(c^*_-(\underline{L_{+}})+\varepsilon_0)\mu_0}<1$ and
			$$
			\underline{L_+}^n[\phi_0](x)\leq \phi_0(x+n(c^*_-(\underline{L_{+}})+\varepsilon_0)),
			\,  \,  \forall  (n,x)\in \mathbb{N}\times \mathbb{R},
			$$
			where $\phi_0(x)=e^{\mu_0 x},\,  \,  \forall x\in \mathbb{R}$.
			
			For each $s>0$,  let $A_s=\inf\{A>0:  W(x)\leq(Ae^{\mu_0 x}+s),
			\,   \forall x\in \mathbb{R}\}$. Then $A_s\in(0,\infty)$ and $A_s$ is nonincreasing in $s\in [0,\infty)$. Let $A_0=\lim\limits_{s\rightarrow 0}A_s$. Then $A_0\in (0,\infty]$.  We further have the following
			claim.
			
			{\bf Claim. } $A_0<\infty$.
			
			Otherwise, $A_0=\infty$. By the conditions of $W(\pm\infty)$ and the definition of $A_s$, it follows that there exist sequences $\{s_k\}^\infty_{k=1}$ in $(0,1)$ and $\{x_k\}^\infty_{k=1}$ in $\mathbb{R}$ such that $\lim\limits_{k\to\infty}s_k=0$,  $\lim\limits_{k\to\infty}A_{s_k}=\infty$, and $W(x_k)=(A_{s_k}e^{\mu_0 x_k}+s_k)$.  Note that
			$\sup\limits_{k\in\mathbb{N}}x_k<\infty$ due to the boundedness of $W$. 	Without loss of generality, we may assume that $x_k\rightarrow x_0\in[-\infty,\infty)$ as $k\rightarrow \infty$. Next we proceed according to  two cases.
			
			{\bf{Case 1.}} $x_k\rightarrow x_0\in \mathbb{R}$ as $k\rightarrow \infty$.
			
			In this case, letting $k\to \infty$, we have $A_{0}=\frac{W(x_0)}{e^{\mu x_0}}\in \mathbb{R}_+$, a contradiction.	
			
			{\bf{Case 2.}} $x_k\rightarrow -\infty$ as $k\rightarrow \infty$.
			
			Let $v_k(x)=A_{s_k}e^{\mu_0 x} +s_k-W(x), \,  \,  \forall x\in \mathbb{R}, k\in \mathbb{N}$.  Take $\xi_0\in (0,\infty)$ and $k_0\in \mathbb{N}$
			such that
			 $$
			 \int_{-\infty}^{-\xi_0}\mathfrak{k}(y){\rm d} y<\frac{\gamma_0}{||\underline{R}||_{L^\infty(\mathbb{R},\mathbb{R})}}, \quad  \underline{R}(y)<2\gamma_0+\limsup\limits_{s \to -\infty}\partial_u f(s,0), \,  \,  \forall
			 y\leq \xi_0+x_{k_0}.
			 $$
			 Since  $\underline{L}[W]\geq Q[W] = W$ and $\underline{L}[\phi_0]\leq \underline{L_+}[\phi_0]\leq \phi_0$,
			 it follows  that  for any $(k,x)\in \mathbb{N}\times  \mathbb{R}$ with $k\geq k_0$ and $x\leq x_{k_0}$, there holds
			\begin{eqnarray*}
				v_k(x)&\geq& \underline{L}[A_{s_k}\phi_0-W](x) +s_k
				\\
				&= &\underline{L}[A_{s_k}\phi_0-W+s_k](x) +s_k-\underline{L}[s_k](x)
				\\
				&\geq &s_k(1-\underline{L}[1](x))
				\\
				&=&s_k[1-\int_{\mathbb{R}}\underline{R}(x-y)\mathfrak{k}(y){\rm d} y]
				\\
				&=&s_k[1-\int_{-\infty}^{-\xi_0}\underline{R}(x-y)\mathfrak{k}(y){\rm d} y-\int_{-\xi_0}^\infty\underline{R}(x-y)\mathfrak{k}(y){\rm d} y]
				\\
				&\geq&s_k[1-||\underline{R}||_{L^\infty(\mathbb{R},\mathbb{R})}\int_{-\infty}^{-\xi_0}\mathfrak{k}(y){\rm d} y-\int_{-\infty}^{x+\xi_0}\underline{R}(y)\mathfrak{k}(x-y){\rm d} y]
				\\
				&\geq&s_k[1-||\underline{R}||_{L^\infty(\mathbb{R},\mathbb{R})}\int_{-\infty}^{-\xi_0}\mathfrak{k}(y){\rm d} y-(2\gamma_0+\limsup\limits_{s \to -\infty}\partial_u f(s,0))\int_{-\infty}^{x+\xi_0}\mathfrak{k}(x-y){\rm d} y]
				\\
				&\geq&s_k[1-3\gamma_0-\limsup\limits_{s \to -\infty}\partial_u f(s,0)]>0.
			\end{eqnarray*}
			Letting $x=x_k$ with $k\geq k_0$, we then obtain have $0=v_k(x_k)>0$, a contradiction.
			
			Thus,   the second conclusion of  (iii) follows from the above claim and Theorem~\ref{thm2.4}.
			
			(iv)  In view of Theorem~\ref{thm2.5},
			we only need to verify {\bf (GAS)}.  Fix $\phi\in C_+\setminus\{0\}$.
			%It follows from  (i) and (iii) that for any $\epsilon\in (0,\frac{\min\{u_-^*,u_+^*\}}{3})$, there exists $l_\epsilon>1$ such that  $$|\psi(x)-u_+^*{\bf {1}_{\mathbb{R}_+}}(x)|<\epsilon $$ for all $|x|\geq  l_\epsilon$ and $\psi\in \omega(\phi)\bigcup \{W\}$.
			Note that  $\{W\}\bigcup \omega(\phi)\subseteq C_+^\circ \cap C_{u_{m_0}^*}$ for some $m_0\in \mathbb{N}$.
			
	 Let $\gamma_0=\frac
	 12 (1+\limsup\limits_{s \to -\infty}\partial_u f(s,0))$.
In view of \cite[Lemma 5.5 ]{yz2025},  there exists $\rho_0>1$ such that
	$$
	f(x,u)\leq \gamma_0u,	\quad \partial_u f(x,u)\leq \gamma_0<1,
		\,  \,  \forall (x,u)\in (-\infty,-\rho_0]\times [0, u_{m_0}^*].
	$$
Define		
$$
a^*=\sup \{ a\in\mathbb{ R}_{+}: \psi(x)\geq a W(x), \,  \,  \forall
 (x,\psi) \in [-\rho_{0},\infty) \times \omega(\phi)\}
$$
and
$$b^*=\sup \{ b\in\mathbb{ R}_{+}: \psi(x)\leq b W(x),\,  \,  \forall
 (x,\psi) \in [-\rho_{0},\infty) \times \omega(\phi)\}.
 $$
By the choice of $a^*,b^*$, we then have
$$
0<a^*\leq 1\leq b^*<\infty,    \quad a^*W|_{[-\rho_{0},\infty)}\leq \omega(\varphi)|_{[-\rho_{0},\infty)}\leq b^*W|_{[-\rho_{0},\infty)}.
$$
			
			We further show that $a^*W\leq \omega(\phi)\leq b^*W$. Otherwise,
			we have either
			$$\inf\{\psi(x)-a^*W(x): (x,\psi)\in\mathbb{R}\times \omega(\phi)\}<0,$$
			 or
			 $$\inf\{b^*W(x)-\psi(x): (x,\psi)\in\mathbb{R}\times \omega(\phi)\}<0.$$ This, together with the choices of $a^*,b^*$ and the conclusions (i), (iii), implies that there exists $(x^*,\psi^*)\in (-\infty,-\rho_0)\times \omega(\phi)$ such that either
			 $$\psi^*(x^*)-a^*W(x^*)=\inf\{\psi(x)-a^*W(x): (x,\psi)\in\mathbb{R}\times \omega(\phi)\}<0,$$
			 or
			 $$b^*W(x^*)-\psi^*(x^*)=\inf\{b^*W(x)-\psi(x): (x,\psi)\in\mathbb{R}\times \omega(\phi)\}<0.
			 $$
			  We only consider the former since the latter can be dealt in a similar way. By the invariance of $\omega(\phi)$, there exists
			  $\psi^{**}\in\omega(\phi)$ such that
			  $\psi^*=Q[\psi^{**};f]$.  It then follows from the choices $W,\psi^*,\psi^{**}$ and  $x^*$ that
			\begin{eqnarray*}
			\psi^*(x^*)-a^*W(x^*)
				&=& Q[\psi^{**};f](x^*)-a^*Q[W;f](x^*)
				\\
				&=&\int_{\mathbb{R}}\mathfrak{k}(y)[f(x^*-y,\psi^{**}(x^*-y))-a^{*}f(x^*-y,W(x^*-y))]{\rm d}y
				\\
				&\ge&\int_{\mathbb{R}}\mathfrak{k}(x^*+\rho_0-z)[f(z-\rho_{0},\psi^{**}(z-\rho_{0}))-f(z-\rho_{0},a^{*}W(z-\rho_{0}))]{\rm d}z
				\\
				&=&\int_{\mathbb{R_{+}}}\mathfrak{k}(x^*+\rho_0-z)[f(z-\rho_{0},\psi^{**}(z-\rho_{0}))-f(z-\rho_{0},a^{*}W(z-\rho_{0}))]{\rm d}z
				\\
				&&+\int_{-\infty}^{0} \mathfrak{k}(x^*+\rho_0-z)[f(z-\rho_{0},\psi^{**}(z-\rho_{0}))-f(z-\rho_{0},a^{*}W(z-\rho_{0}))]{\rm d}z
				\\
				&\ge&\int_{-\infty}^{0} \mathfrak{k}(x^*+\rho_0-z)[f(z-\rho_{0},\psi^{**}(z-\rho_{0}))-f(z-\rho_{0},a^{*}W(z-\rho_{0}))]{\rm d}z
				\\
				&=&\int_{-\infty}^{0} \mathfrak{k}(x^*+\rho_0-z) l(z-\rho_0)(\psi^{**}(z-\rho_{0})- a^*W(z-\rho_0)){\rm d}z
				\\
				&\geq &\int_{-\infty}^{0} \mathfrak{k}(x^*+\rho_0-z)\gamma_0\inf\{\psi(x)-a^*W(x): (x,\psi)\in\mathbb{R}\times \omega(\phi)\} {\rm d}z
				\\
				&=&\gamma_0(\psi^*(x^*)-a^*W(x^*))\int_{-\infty}^{0} \mathfrak{k}(x^*+\rho_0-z) {\rm d}z
				\\
				&\geq &\gamma_0(\psi^*(x^*)-a^*W(x^*)),
			\end{eqnarray*}
			where
			$$
			l(z):=\int_0^1\partial_u f(z,a^*W(z)+\tau(\psi^{**}(z)-a^*W(z)){\rm d} \tau\in [0,\gamma_0] \subseteq [0,1),\,  \,  \forall z\leq -\rho_0.
			$$
			This gives  rise to $\psi^*(x^*)-a^*W(x^*)\geq 0$, a contradiction.
			
			Now it suffices to prove $a^*=b^*=1$.  Otherwise, either $a^*<1$, or $b^*>1$. By  the strictly subhomogeneity of $f$  and  the limits  of $W$
			and $\omega(\phi)$ at $\infty$,  we have
			either $\omega(\phi)-a^*W\subseteq  C_+^\circ$, or $-\omega(\phi)+b^*W\subseteq  C_+^\circ$, which contradicts  the definitions $a^*$ and $b^*$.
			As a result, we obtain $\omega(\phi)=\{W\}$.
				\end{proof}
		
To finish this section, we provide a counterexample to show  that the existence of nontrivial fixed points may be affected by the properties of $f(x,\cdot)$   at finite range of location $x$.
		
Let us introduce three functions:
$$\underline{k}(x)=\frac{e^{-(x-2)^2}}{\sqrt{\pi}},\quad  \underline{\gamma}(x)= e^{-x^2},\quad  \forall x\in \mathbb{R},
$$
and
$$
\underline{f}(u)=\max\{ue^{-u}, \frac{1}{e}\},\quad  \forall
u\in \mathbb{R}_+.
$$
By a straightforward computation, we then have the following result.
 \begin{lemma}\label{lemma4.2}
			Let $\underline{k},\underline{\gamma}$ and $\underline{f}$ be defined as above. Then the following statements are valid:
			\begin{description}
			\item [(i)] $\underline{k}(\mathbb{R})\subseteq (0,1)$, $\int_{ \mathbb{R}}\underline{k}(y){\rm{d}} y=1$, and $\inf_{\mu>0}\{\frac{1}{\mu}\ln[e\int_{ \mathbb{R}}e^{-\mu y}\underline{k}(y){\rm{d}} y]\}=-1<0;$
			\item [(ii)]  $\lim\limits_{|x|\to \infty}\gamma(x)=0$ and $\gamma(\mathbb{R})\subseteq (0,1]$;
						\item [(iii)]  $\underline{f}(0)=0$, $\underline{f}'(0)=1$,  $\underline{f}(\mathbb{R})=[0,\frac{1}{e}]$, and $\underline{f}$ is nondecreasing, strictly subhomogeneous on $(0,\infty)$.
		\end{description}
		\end{lemma}
		
		For any given number $\beta>0$, we  define the linear operator $L_\beta$
		on  $C\rightarrow C$ as follows:
		$$
		L_\beta[\phi](x)=\int_{\mathbb{R}}\beta \underline{\gamma}(y) \phi(y) \underline{k}(x-y){\rm{d}} y,\quad  \forall (x, \phi)\in \mathbb{R}\times C.
		$$
		
		\begin{lemma}\label{lemma4.3}
			Let $L_\beta$ be defined as above. Then the following statements are valid:
			\begin{description}
				\item [(i)] $L_\beta=\beta L_1$ for all $\beta \in (0, \infty)$.
				\item[(ii)] For any $\beta \in (0, \infty)$, $L_\beta[C_+\setminus\{0\}]\subseteq C_+^\circ$, $L_\beta$ is a bounded linear  operator on $(C, |\cdot|_\infty)$, and $L_\beta|_{C_1}:C_1\to C$ is a continuous and compact map on $C_1$.
				\item [(iii)] $\rho_{L_1}>0$ and $\rho_{L_\beta}=\beta\rho_{L_1}$, where $\rho_L$ represents the asymptotic spectral radius of $L$, as defined in \cite{yz2023}.
			\end{description}
			Hence, there exists $\beta_0 \in (0, \infty)$ such that $\rho_{L_{\beta}}>1$ for all $\beta\geq \beta_0$.
		\end{lemma}
		
We further define two functions $g(x,u)$ and $h(x,u)$ as follows:
$$
g(x, u)=\beta_0\underline{\gamma}(x)\underline{f}(u),
\quad
h(x,u)=\max \{\beta_0\underline{\gamma}(x),e-e^{-x}\}\underline{f}(u),
\quad  \forall (x,u)\in  \mathbb{R}\times  \mathbb{R}_+.
$$
\begin{lemma}\label{lemma4.4}
			Let $g$ and  $h$ be defined as above. Then $h\geq g, \, \, \frac{\partial g}{\partial u}(x,0)=\beta_0\underline{\gamma}(x),\, \,
			 \frac{\partial h}{\partial u}(x,0)=\max\{\frac{\partial g}{\partial u}(x,0),e-e^{-x}\}$, \,  $\lim\limits_{x\rightarrow-\infty}\frac{\partial h}{\partial u}(x,0)=0$, \, and $\lim\limits_{x\rightarrow\infty}\frac{\partial h}{\partial u}(x,0)=e$.
		\end{lemma}
		
By  \cite[Theorem 4.9-(ii)]{yz2023} and the choices of $\beta_0, \underline{f}$ and $h$, we easily prove the  following result.
		
		\begin{prop}\label{prop4.3}
			There exists $\underline{W}\in C_+^\circ$ such that $Q[\underline{W};g]=\underline{W}$ and $\underline{W}(\pm \infty)=0$.
		\end{prop}
		
		\begin{prop}\label{prop4.4}The following statements are valid:
			\begin{description}
				\item [(i)] $c^*_-(\infty; Q[\cdot;h])=-1<0$ and $c^*_+(\infty; Q[\cdot;h])>0$
				\item[(ii)]
				There exists $\overline{W}\in \underline{W}+C_+$ such that $Q[\overline{W}; h]=\overline{W},\, \,
				\overline{W}(-\infty)=0$, and $\overline{W}(\infty)=1$,  where $\underline{W}$ is defined as in Proposition~\ref{prop4.3}.
			\end{description}
		\end{prop}

		\begin{proof}
			(i) follows from Lemma~\ref{lemma4.2}-(ii).
			
			(ii)  Let $\mathfrak{K}:=[\underline{W},\max\{\beta_0,||\underline{W}||_{L^\infty(\mathbb{R},\mathbb{R})}\}]_C$. By the choices of $g$, $h$ and $\underline{W}$, we can verify
			that  $Q[\mathfrak{K};h]\subseteq \mathfrak{K}$.   Thus, the  Schauder fixed point Theorem implies  that $Q[\cdot; h]$ has  a fixed point $\overline{W}$ in $\mathfrak{K}$.
		In view of  the former conclusion of Theorem~\ref{thm4.1}-(i), we then obtain $\overline{W}(-\infty)=0$ and $\overline{W}(\infty)=1$.
			\end{proof}
		
			Based on the definition of $h$, we find that $Q[\cdot;h]$ satisfies all the conditions in the second part of Theorem~\ref{thm4.1}-(iv) except for the limiting linear control condition that $f(s,u)\leq \frac{{\rm d}  f_+(0)}{{\rm d} u}u$ for all $(s,u)\in \mathbb{R}\times \mathbb{R}_+$.  Proposition~\ref{prop4.4} indicates that  the nontrivial  fixed point may exist when  the nonlinearity function does not satisfies the limiting linear control condition.  In other words,  the  value of the nonlinear function  at the finite range of location $x$  may lead to the existence of nontrivial fixed points.
		
	\section*{Appendix. Orbit precompactness for noncompact maps}
	
	In a series of three papers \cite{yz2020, yz2023, yz2025},  we obtained
	the  existence of fixed points for  monotone maps with asymptotic
	translation invariance under the assumption that some forward orbits are precompact with respect to the compact open topology (i.e., the local uniform convergence topology). In this Appendix, we provide a sufficient condition for a class of noncompact maps to admit this orbit precompactness.
	
	Let $\mathbb{Z}$,
	$\mathbb{N}$, $\mathbb{R}$, $\mathbb{R}_+$, $\mathbb{R}^N$, and $\mathbb{R}_+^N$ be the sets of all
	integers, nonnegative integers, reals,   nonnegative reals,  N-dimensional real vectors,  and  N-dimensional  nonnegative real vectors,
	respectively. We equip  $\mathbb{R}^N$ with the norm $||\xi||_{\mathbb{R}^N}\triangleq
	\sqrt{\sum \limits_{n=1}^{N}\xi_n^2}$.
	Let $\hat{1}=(1,1,...,1)^T\in \mathbb{R}^N$, $X=BC(\mathbb{R},\mathbb{R}^N)$ be the normed
	vector space of all bounded and continuous functions from
	$\mathbb{R}$ to $\mathbb{R}^N$ with the norm $||\phi||_{X}\triangleq
	\sum \limits_{n=1}^{\infty}2^{-n}\sup\limits_{|x|\leq n}\{||\phi(x)||_{\mathbb{R}^N}\}$, $X_+=\{\phi\in
	X:\phi(x)\in \mathbb{R}_+^N, \forall x\in \mathbb{R}\}$.
	For a compact metric space $M$, let
	$Y=C(M, \mathbb{R}^N),~Y_+=C(M, \mathbb{R}_+^N)$,  and $||\beta||_{Y}\triangleq  \sup\limits_{\theta\in
		M} \{||\beta(\theta)||_{\mathbb{R}^N}\}$.
	Let  $\mathcal{C}=C(M,X)$ be the
	normed vector space of all continuous maps from $M$
	into $X$ with the norm $||\varphi||_{\mathcal{C}}\triangleq
	\sup\limits_{\theta\in M}\{||\varphi(\theta)||_{X}\}$, and $\mathcal{C}_+=C(M,X_+)$. It follows that $\mathcal{C}_+$ is a closed
	cone in the normed vector space $\mathcal{C}$.
	
	For convenience, we also use notations $\mathcal{C}=BC(M\times \mathbb{R},\mathbb{R}^N)$, $\mathcal{C}_+=BC(M\times \mathbb{R},\mathbb{R}^N_+)$. We define $\mathcal{C}_{\phi,\psi}:=[\phi,\psi]_{\mathcal{C}}$
	for any $\phi \leq \psi$ in $\mathcal{C}$, and $\mathcal{C}_{\phi}:=
	[0,\phi]_{\mathcal{C}}$ whenever $\phi\in \mathcal{C}_+$.
	Let $d_{\mathcal{C}}$ be the metric
	induced by  $||\cdot||_{\mathcal{C}}$ in $\mathcal{C}$, and
	$(\mathcal{C},\|\cdot\|_{\infty})$ be the Banach space equipped with the supremum norm $\|\cdot\|_{\infty}$.
	
	Recall that for any given $y\in \mathbb{R}$, the translation
	operator $T_y$ on $\mathcal{C}$ is defined by $[T_y\phi](\theta,x)=
	\phi(\theta,x-y)$ for all $\phi\in \mathcal{C}$ and $(\theta,x)\in
	M\times\mathbb{R}$.
	In what follows, a map $M:  (\mathcal{C}_+, d_{\mathcal{C}})\rightarrow  (\mathcal{C}_+, d_{\mathcal{C}})$ is said to be continuous if  $M: (\mathcal{C}_{r},d_{\mathcal{C}})\rightarrow  (\mathcal{C}_+, d_{\mathcal{C}})$ is  continuous for any given $r\in Int(Y_+)$.
	
	\
	
	\noindent
	{\bf Theorem A. }  {\it
		Let $Q: (\mathcal{C}_+, d_{\mathcal{C}})\rightarrow  (\mathcal{C}_+, d_{\mathcal{C}})$ be a continuous map such that $Q(\mathcal{C}_{r^*})\subseteq \mathcal{C}_{r^*}$ for some
		$r^*\in Int(Y_+)$.  Assume that $Q=L+K$ satisfies the following two conditions:
		\begin{enumerate}
			\item[(1)] $L: (\mathcal{C}_+,\|\cdot\|_{\infty})
			\rightarrow (\mathcal{C}_+,\|\cdot\|_{\infty})$
			a Lipschitz map with its Lipschitz constant $k\in (0,1)$;
			
			\item[(2)] $K: (\mathcal{C}_+, d_{\mathcal{C}})\rightarrow  (\mathcal{C}_+, d_{\mathcal{C}})$ is a continuous map,
			and $K(\mathcal{C}_{r^*})$ is precompact in $(\mathcal{C}, ||\cdot||_{\mathcal{C}})$.
		\end{enumerate}
		Then for any $y\in \mathbb{R}$, the forward orbit $\{(T_yQ)^n(r^*)\}_{n \ge 1}$ is precompact
		in $(\mathcal{C},\|\cdot\|_{\mathcal{C}})$.}
	
	\begin{proof}
		Define $P_1=K$, and $P_{n+1}=L^nK+P_nQ, \forall n\ge 1$.   By induction, it is easy to verify that
		$$
		Q^n=L^n+P_n, \quad
		P_n(\mathcal{C}_+)\subseteq \mathcal{C}_+,  \quad
		L^n(\mathcal{C}_{r^*})\subseteq L^n(0)+\mathcal{C}_{-k^n\|r^*\|_Y\hat{1},k^n
			\|r^*\|_Y\hat{1}}, \quad \forall n\geq 1.
		$$
		Thus,  each set $P_n(\mathcal{C}_{r^*})+L^{n}(0)
		\subseteq \mathcal{C}_{-\|r^*\|_{Y} \hat{1},r^*+\|r^*\|_{Y} \hat{1}}$ is precompact in  $(\mathcal{C},\|\cdot\|_{\mathcal{C}})$.
		Let $\alpha$ be the Kuratowski measure of non-compactness for the complete metric space $(\mathcal{C}_{-2\|r^*\|_{Y}\hat{1},2\|r^*\|_{Y} \hat{1}},d_{\mathcal{C}})$. It then follows that
		\begin{eqnarray*}
			\alpha(Q ^n(\mathcal{C}_{r^*})) &\le& \alpha(L^n(\mathcal{C}_{r^*})-L^{n}(0))+
			\alpha(P_n(\mathcal{C}_{r^*}) +L^{n}(0))\\
			&=& \alpha(L^n(\mathcal{C}_{r^*})-L^{n}(0))\\
			&\le& \alpha(\mathcal{C}_{-k^n\times \|r^*\|_Y \hat{1},
				k^n\times \|r^*\|_Y \hat{1}})\\
			&=& k^n\|r^*\|_Y\alpha(\mathcal{C}_{-\hat{1},\hat{1}}), \quad \forall n\geq 1,
		\end{eqnarray*}
		and hence, $\lim\limits_{n\to \infty}\alpha(\overline{Q^n(\mathcal{C}_{r^*})})=\lim\limits_{n\to \infty}\alpha(Q^n(\mathcal{C}_{r^*}))=0$.
		
		In view of $Q(\mathcal{C}_{r^*})\subseteq \mathcal{C}_{r^*}$, we have
		$\overline{Q^n(\mathcal{C}_{r^*})}\subseteq \overline{Q^{n-1}(\mathcal{C}_{r^*})}, \forall n\ge 1$. By an elementary property of the noncompactness measure
		(see, e.g., \cite[Lemma 2.1 (a)]{MagalZhao2005}), it follows that the set $A:=\bigcap\limits_{n\ge 1}\overline{Q^n(\mathcal{C}_{r^*})}$ is nonempty and compact in $(\mathcal{C}_{-2\|r^*\|_{Y} \hat{1}, 2\|r^*\|_{Y} \hat{1}},d_{\mathcal{C}})$, and
		$\lim\limits_{n\to \infty}\sup_{\phi\in A_n}d_{\mathcal{C}}(\phi,A)=0$,
		where  $A_n=\overline{Q^n(\mathcal{C}_{r^*})}$. This implies that
		$\lim\limits_{n\to \infty}d_{\mathcal{C}}(Q^n(r^*),A)=0$. Thus,
		the forward orbit $\{Q^n(r^*)\}_{n\ge 1}$ is precompact in $(\mathcal{C},\|\cdot\|_{\mathcal{C}})$. For any given $y\in \mathbb{R}$,  let $\hat{Q}=T_yQ$, $\hat{L}=T_yL$, and $\hat{K}=T_yK$. Clearly,
		$\hat{Q}=\hat{L}+\hat{K}$ also satisfies the above conditions (1) and (2). It then
		follows that the forward orbit $\{\hat{Q}^n(r^*)\}_{n\ge 1}=\{(T_yQ)^n(r^*)\}_{n \ge 1}$ is precompact in $(\mathcal{C},\|\cdot\|_{\mathcal{C}})$.
	\end{proof}
	
	\
	
		\noindent
			{\bf Acknowledgements.} T.  Yi's research is supported by the National Natural Science Foundation of  China (NSFC 12231008),  and X.-Q. Zhao's research is supported in part by the NSERC of Canada (RGPIN-2019-05648 and RGPIN-2025-04963).

\end{document}